\documentclass[a4paper, reqno]{amsart}
\usepackage{graphicx} % Required for inserting images

\usepackage[margin=3cm]{geometry}
\usepackage{latexsym}
\usepackage{tcolorbox}
\usepackage{soul}
\usepackage{etoolbox}
\usepackage[normalem]{ulem}

\usepackage[utf8]{inputenc}
\usepackage[T1]{fontenc}

\usepackage{amssymb, amsmath, amsthm, amsrefs, xcolor}
\usepackage{mathtools}
\usepackage{hyperref}
\usepackage{mathrsfs}
\usepackage[matha, mathx]{mathabx}
\usepackage[shortlabels]{enumitem}

\newtheorem{thm}{Theorem}[section]
\newtheorem{corollary}[thm]{Corollary}
\newtheorem{lemma}[thm]{Lemma}
\newtheorem{proposition}[thm]{Proposition}

\theoremstyle{definition}

\newtheorem{example}[thm]{Example}
\newtheorem{remark}[thm]{Remark}

\newtheorem*{Que}{Question}

\newcommand{\R}{{\mathbb R}}
\newcommand{\N}{{\mathbb N}}
\newcommand{\C}{{\mathbb C}}

\newcommand{\dd}{{\mathrm d}}

\newcommand{\mG}{\mathcal{G}}

\newcommand{\mB}{\mathcal{B}}
\newcommand{\mS}{\mathcal{S}}
\newcommand{\mA}{\mathcal{A}}

\newcommand{\Fm}{\mathcal{A}}
\newcommand{\mF}{\mathcal{F}}
\newcommand{\mK}{\mathcal{K}}
\newcommand{\mQ}{\mathcal{Q}}
\newcommand{\mM}{\mathcal{M}}
\newcommand{\lm}{\lambda}
\newcommand{\LCrad}{\mathfrak{L}_d^{\textrm rad}}
\newcommand{\LC}{\mathfrak L_d}

\newcommand{\Prad}{\mathfrak P_d}

\DeclareMathOperator{\supp}{supp}

\begin{document}

\numberwithin{equation}{section}
\title[High-dimensional log-concave densities]{High-dimensional limits and extremizers for maximal functions associated with log-concave densities}

\author{Valentina Ciccone}
\address[Valentina Ciccone]{Institute of Mathematics of the Polish Academy of Sciences, Śniadeckich 8, 00-656 Warszawa, Poland}
\email{vciccone@impan.pl}

\author{B{\l}a{\.z}ej Wr{\'o}bel}
\address[B{\l}a{\.z}ej Wr{\'o}bel]{
Institute of Mathematics
	of the Polish Academy of Sciences\\
	Śniadeckich 8\\
	00-656 Warszawa\\
	Poland \& Institute of Mathematics\\
	University of Wroc{\l}aw\\
	Plac Grun\-waldzki 2\\
	50-384 Wroc{\l}aw\\
	Poland}
\email{blazej.wrobel@math.uni.wroc.pl}

\thanks{Both authors were supported by the National Science Centre, Poland, grant Sonata Bis 2022/46/E/ST1/00036.}

\begin{abstract}

We introduce a unified framework to establish the high-dimensional asymptotic behavior of maximal functions associated with radial log-concave probability densities, encompassing the maximal heat semigroup, Hardy-Littlewood maximal function over Euclidean balls, and, additionally, maximal spherical means. Namely, for any $p \in (1, \infty)$, we prove that the $L^p(\mathbb{R}^d)$ operator norms of these maximal operators all converge as the dimension $d \to \infty$ to a single, universal limit $\lambda(p)$. 

Furthermore, by proving that the $L^p$ operator norms for the heat semigroup $\mathcal G_*^d$ are monotonically non-decreasing in the dimension, we provide explicit quantitative bounds on the universal limit, showing that $\frac{2}{5}\frac{p}{p-1} \le \|\mathcal{G}_*^1\|_{L^p(\mathbb{R}) \to L^p(\mathbb{R})} \le \lambda(p) \le \frac{p}{p-1}$. 

We also prove an extremality property: among all symmetric convex bodies in high dimensions, the maximal operator associated with the Euclidean ball achieves the asymptotically minimal $L^p$ operator norm. 

Our main results are established via a general transference principle that allows us to control maximal functions via Fourier multiplier symbols. To estimate these symbols uniformly across log-concave densities, we import variance type bounds and thin-shell type concentration of measure results, which are novel tools in the study of maximal functions. In particular, to prove the extremality property, we require a variance type bound for general log concave measures established in a recent series of breakthroughs in high dimensional convex geometry.

\end{abstract}

\maketitle

\section{Introduction}

Let $K$ be a probability density on $\mathbb{R}^d.$ We say that $K$ is log-concave if its support, $\supp(K)=\lbrace x\in\mathbb{R}^d:\, K(x)>0 \rbrace$, is a convex set, and $\log K$ is a concave function over 
$\supp(K)$. $K$ is called symmetric if $K(x)=K(-x)$ for every $x\in\mathbb{R}^d$. For $t>0$, and for every $x\in\mathbb{R}^d$, we define the Hardy--Littlewood averaging operator associated with a symmetric, log-concave density $K$ by
\begin{equation}
\label{eq:mKt}
    \mK_t f(x)= t^{-d}\int_{\R^d}K(t^{-1}y)f(x-y)\dd y~, \qquad f\in L^1_{loc}(\mathbb{R}^d)~.
\end{equation}
Important examples of such operators include 
the heat-semigroup 
\[
\mG_tf(x)= 
\frac{1}{(2\pi t^2)^{d/2}}\int_{\R^d}e^{-|y|^2/(2t^2)} f(x-y)\dd y~,
\]
the Hardy-Littlewood averages over balls $B_t=\{x\in \R^d\colon |x|\le t\}$ and over symmetric cubes $Q_t=[-t/2,t/2]^d$
\begin{align*}
    \mB_tf(x)= \frac{1}{|B_t|}\int_{B_t}f(x-y)\dd y~, \qquad 
\mQ_t f(x)=\frac{1}{|Q_t|}\int_{Q_t}f(x-y)\dd y~, 
\end{align*}
and, more generally, the Hardy--Littlewood averages over an arbitrary symmetric convex body $G\subset \mathbb{R}^d$ 
\[
\mK^G_t f(x) = \frac{1}{|G_t|}\int_{G_t}f(x-y)\dd y~, \qquad G_t=\lbrace x\in\mathbb{R}^d:\, t^{-1}x\in G \rbrace~.
\]

Let $\mathcal{A}_t$ be any of the above averaging operators, or any difference thereof. We write $\mA_*$ for the corresponding maximal operator,
\[
\mA_* f(x)=\sup_{t>0}|\mA_t f(x)|.
\]
For each  $p\in (1,\infty]$ we denote by
$\|\mA_*\|_{L^p(\R^d)}$ the operator norm on $L^p(\R^d)$ of $\mA_*$; that is, $\|\mA_*\|_{L^p(\R^d)}$ equals the best constant  $C$ in the inequality
\[
\Vert \sup_{t>0} |\mA_t f|\Vert_{L^p(\R^d)}\le C\|f\|_{L^p(\R^d)}.
\]
It is well known that if $K$ is a symmetric log-concave probability density on $\mathbb{R}^d$, then the associated maximal operator $\mK_\ast$ is bounded on $L^p(\mathbb{R}^d)$ for all $d\geq 1$, $p\in (1,\infty]$. In other words, $\|\mK_*\|_{L^p(\R^d)} $ is finite for all $d\in \N$ and $p\in (1,\infty].$  
We say that a symmetric probability density $K$ on $\mathbb{R}^d$ is isotropic with variance $\sigma^2$ if 
$$\int_{\mathbb{R}^d} x K(x) \dd x=0~, \qquad \int_{\mathbb{R}^d} \langle x,y\rangle^2 K(x) \dd x =\sigma^2 |y|^2 \quad \forall\; y\in\mathbb{R}^d~. $$
We stress that, for the purpose of estimating the operator norms, there is no loss of generality in restricting to isotropic, log-concave densities. In fact, any symmetric probability density may be rendered isotropic by a suitable invertible linear transformation, and this does not affect the value of the operator norm on $L^p(\mathbb{R}^d)$ of the associated maximal function.

If the density $K$ is radial and log-concave, then it must be radially decreasing. Thus, one may prove the pointwise inequality
\begin{equation}
\label{eq: point est}
    \sup_{t>0} |\mK_tf(x)| \; \leq \sup_{t>0} \; |\mB_tf(x)|,
\end{equation}
from which it immediately follows that
\begin{align}\label{eq_gauss_leq_ball_leq_sphere}
   \|\mK_*\|_{L^p(\R^d)} \, \leq \,  \|\mB_*\|_{L^p(\R^d)}.
\end{align}
Our first main goal in this work is to show that, in the limit as $d$ tends to infinity, \eqref{eq_gauss_leq_ball_leq_sphere} becomes an equality uniformly for all radial log-concave probability densities. This will be a consequence of the stronger statement in \eqref{eq:thm_rad_logconc_diff_lim}, which concerns the behavior of the maximal functions corresponding to the difference of a radial log concave isotropic probability density and the  isotropic Gaussian density with the same variance. 

In what follows, we let $(\mathcal{G}^\sigma)_t:=\mathcal{G}_{\sigma t},$ $t>0,$ so that the isotropic Gaussian kernel corresponding to $(\mathcal{G}^\sigma)_t$ has variance $\sigma^2 t^2.$  We are now ready to state our first theorem.

\begin{thm}\label{thm_rad_logconc}
    Let $K$  be any
    radial, log-concave probability density on $\mathbb{R}^d$ which is isotropic with variance $\sigma^2$.
    For each $p\in (1,\infty)$ there exist two positive constants $C_p,\alpha_p$ such that
    \begin{equation}\label{eq:thm_rad_logconc_diff_lim}
    \|(\mathcal{K}-\mathcal{G}^\sigma)_\ast\|_{L^p(\mathbb{R}^d)} \leq C_pd^{-\alpha_p}~.
    \end{equation}
Furthermore, for each $p\in (1,\infty)$  the limit $\lim_{d\rightarrow\infty}\|\mathcal{G}_\ast\|_{L^p(\mathbb{R}^d)}=:\lm(p)$ exists and
  \begin{equation}
  \label{eq:thm_rad_logconc_lim}
      \lim_{d\to \infty}\sup_{K\in \LCrad} \big| \|\mathcal{K}_\ast\|_{L^p(\mathbb{R}^d)}-\lambda(p)\big|=0,
    \end{equation}
    where $\LCrad$ denotes the set of all radial log-concave probability densities on $\R^d$.
Additionally, $\lm(p)$ satisfies  
    \begin{equation}
    \label{eq: lm bounds}
   \max\left\lbrace \frac25 \frac{p}{p-1}, 1 \right\rbrace \le \|\mG_\ast^1\|_{L^p(\mathbb{R})}\le \lm(p)\le \frac{p}{p-1}~. 
    \end{equation}
\end{thm}

Equation \eqref{eq:thm_rad_logconc_lim} asserts that, in the high-dimensional limit $d\rightarrow\infty$, the operator norms of maximal functions associated with radial log-concave densities are asymptotically equal to the same quantity $\lm(p)$.
The following corollary is an immediate consequence of Theorem \ref{thm_rad_logconc}.
\begin{corollary} For each $p\in (1,\infty)$
    \begin{equation}\label{eq:thm_rad_diff_balls_gauss}
       \|(\mathcal{B}-{\mathcal{G}^{1/\sqrt{d+2}}})_\ast\|_{L^p(\mathbb{R}^d)} \leq C_pd^{-\alpha_p}~,
    \end{equation}
    and, consequently,
\begin{align*}
    \lim_{d\rightarrow \infty} \|\mathcal{G}_\ast\|_{L^p(\mathbb{R}^d)}=\lim_{d\rightarrow \infty} \|\mathcal{B}_\ast\|_{L^p(\mathbb{R}^d)}
    =\lambda(p)~.
\end{align*}
\end{corollary}

Our next result shows that the same convergence result holds for the $L^p(\mathbb{R}^d)$-operator norm of the spherical maximal function. For $t>0$ and every $x\in\mathbb{R}^d$, the spherical averaging operator is given by
$$\mathcal{S}_tf(x)=\frac{1}{\sigma(\mathbb{S}^{d-1})}\int_{\mathbb{S}^{d-1}}f(x-t\omega)\dd \sigma(\omega)~, \qquad f\in\mathcal{S}(\mathbb{R}^d)~,$$
where $\sigma$ is the surface measure of the $(d-1)$-dimensional unit sphere. The associated spherical maximal function is 
$$\mathcal{S}_\ast f(x)=\sup_{t>0} \; |\mathcal{S}_tf(x)|~.$$
It is known that $\|\mS_*\|_{L^p(\R^d)}$ is finite for $p> d/(d-1).$ This is due to Stein \cite{St_spher} in dimensions $d\ge 3$ and due to Bourgain \cite{Bo1} in dimension $d=2.$ We remark that for each fixed $p\in (1,\infty)$ we have $p>d/(d-1)$ when $d$ is large enough and in particular $\|\mS_*\|_{L^p(\R^d)}$ is finite when $d\to \infty.$
 The spherical maximal function controls pointwise any maximal function associated with a radial probability density. In particular, $\mB_\ast f(x)\leq \mS_\ast f(x)$ and, therefore,
\begin{align}\label{eq:sphere_geq_ball}
      \|\mB_*\|_{L^p(\R^d)} \, \leq \,  \|\mS_*\|_{L^p(\R^d)}.
\end{align}
Our next Theorem establishes that, in the limit as $d$ tends to infinity, \eqref{eq:sphere_geq_ball} becomes an equality.
\begin{thm}\label{thm:gaussian_sphere}
Let $p\in (1,\infty).$ Then there exists constants $C_p,\alpha_p>0$ and $d_0(p)>1$ depending only on p and such that
    \begin{equation}\label{eq:thm_spher_diff_lim}
      \|(\mathcal{S}-\mG^{\frac{1}{\sqrt{d}}})_\ast\|_{L^p(\mathbb{R}^d)} \leq C_pd^{-\alpha_p},\qquad d>d_0(p).
    \end{equation}
    Consequently,
    \begin{align}\label{eq:limitS}
    \lim_{d\rightarrow\infty}\|\mathcal{S}_\ast\|_{L^p(\mathbb{R}^d)}=\lambda(p)~.
    \end{align}
\end{thm}
Our second main result asserts that $\lambda(p)$ provides a lower bound, as $d\rightarrow\infty$, for the operator norm on $L^p(\mathbb{R}^d)$ of any maximal function associated with a (symmetric) log-concave density.
\begin{thm}\label{thm:extremal}
Let  $p\in(1,\infty).$ We have
\begin{align}\label{eq:liminf_any_symmlogconcave}
    \lambda(p)\leq \liminf_{d\rightarrow \infty} \inf_{K\in {\LC}} \|\mathcal{K}_\ast\|_{L^p(\mathbb{R}^d)}
\end{align}
where ${\LC}$ denotes the family of symmetric log-concave probability densities on $\mathbb{R}^d$. In particular,
 among maximal functions associated to symmetric convex bodies, the norm of the maximal functions corresponding to the balls is asymptotically minimal, as $d\to \infty.$
\end{thm}

To prove Theorem \ref{thm:extremal} we introduce an auxiliary family of convolution and maximal operators that are associated with spherical averages of symmetric log-concave densities. For a kernel $K\in L^1(\mathbb{R}^d)$ we define its spherical average as
\begin{equation} 
\label{eq:spherA}
A^K(x)=\int_{SO(d)}K(Rx)\dd R= \frac{1}{\sigma(\mathbb{S}^{d-1})}\int_{\mathbb{S}^{d-1}}K(|x|u)\dd \sigma(u)~.
\end{equation}
Here, we are considering $SO(d)$ endowed with the normalized Haar measure. Clearly, when $K$ is radial $A^K=K$.
However, when $K$ is an arbitrary log-concave density, the property of being log-concave may not be preserved by spherical averaging. 

\begin{thm}\label{thm_av_logconc}
    Let $A^K$ be the spherical average \eqref{eq:spherA} of a symmetric log-concave probability density $K$. We assume $K$ to be isotropic with variance $\sigma^2$. Then, for each $p\in(1,\infty)$ there exist constants $C_p>0$ and $\alpha_p>0$ such that
    \begin{equation}\label{eq:thm_av_logconc_diff_lim}
      \|(\mathcal{A}^K-\mathcal{G}^\sigma)_\ast\|_{L^p(\mathbb{R}^d)} \le C_p d^{-\alpha_p}.
    \end{equation}
   Furthermore, for each $p\in (1,\infty)$ 
    \begin{equation}\label{eq:thm_av_logconc_lim}
     \lim_{d\rightarrow\infty} \sup_{K\in{\LC}} \big| \|\mathcal{A}^K_\ast\|_{L^p(\mathbb{R}^d)}-\lambda(p)\big|=0~. \\
       \end{equation}
\end{thm}
We obtain \eqref{eq:thm_rad_logconc_diff_lim} in Theorem \ref{thm_rad_logconc} as a particular case of \eqref{eq:thm_av_logconc_diff_lim}. Moreover, Theorem \ref{thm:extremal} will be a consequence of \eqref{eq:thm_av_logconc_lim} combined with the fact that
$$\|\mathcal{A}^K_\ast\|_{L^p(\mathbb{R}^d)}\leq \|\mathcal{K}_\ast\|_{L^p(\mathbb{R}^d)}~.$$

We prove the existence of the limits in Theorem \ref{thm_rad_logconc} as well as the lower bounds on $\lambda(p)$ in \eqref{eq: lm bounds} as a consequence of the following result for the $L^p(\mathbb{R}^d)$-operator norm of the Gaussian maximal function. The upper bound on $\lambda(p)$ in \eqref{eq: lm bounds} is a consequence of the known bound \eqref{eq: mG p/(p-1)}.

\begin{thm}
    \label{propo:GausMonotonicity}
    For each $p\in(1,\infty)$ and $d\geq 1$ we have that
    \begin{align}\label{eq:gauss_norm_mono_intro}
        \|\mG_\ast^{d+1}\|_{L^p(\mathbb{R}^{d+1})}\geq \|\mG_\ast^{d}\|_{L^p(\mathbb{R}^{d})}~.
    \end{align}
    Moreover, there exist $0.4 <c\leq 1$ such that $\|\mG_\ast^{1}\|_{L^p(\mathbb{R})}\geq \max \left\lbrace c \frac{p}{p-1} , 1 \right\rbrace $.  In particular, $\|\mG_\ast^{1}\|_{L^p(\mathbb{R})}> \frac{2}{5} \frac{p}{p-1}>1$ for all $p\in (1,\frac{5}{3})$.
\end{thm}

\begin{remark}
\label{rem: TensMax}
%We stress that 
The proof of \eqref{eq:gauss_norm_mono_intro} relies solely on the fact that the Gaussian kernel is positive, even, and of tensor product type.
In particular, the same monotonicity property holds for any maximal function whose convolution kernel is positive, even, and of tensor product type; an example of this is the centered maximal function associated with cubes in $\mathbb{R}^d$.
\end{remark}

The tools we utilize to prove Theorem \ref{thm_av_logconc} include a general transference principle which allows us to control maximal functions via Fourier multiplier symbols, and variance and thin-shell type bounds for isotropic log-concave random vectors.

Our transference principle is the content of Theorem \ref{general_thm} below. 
\begin{thm}\label{general_thm}
    Let $a$ be a bounded,  almost everywhere differentiable function on $\mathbb{R}^d$. Assume that there exist constants $\sigma>0$ and $J>0$
such that the following inequalities:
\begin{align}\label{eq_general_thm_pointwise_estimate_multiplier_1}
    |a(\xi)|\leq J \sigma|\xi| ~,
\qquad
    |a(\xi)|\leq J \frac{1}{\sigma|\xi|} ~,
\end{align}
\begin{align}\label{eq_general_thm_pointwise_estimate_multiplier_gradient}
    \vert \langle \xi, \nabla a(\xi)\rangle \vert \leq J ~,
\end{align}
hold for almost every $\xi \in \R^d.$ Consider the family of Fourier multiplier operators $\lbrace \Fm_t\rbrace_{t>0}$ given by
    $$\widehat{\Fm_t f}(\xi)=a(t\xi)\widehat{f}(\xi)~.$$
Then there exists a constant $C>0$ depending only on $J$ and such that 
    \begin{equation}
    \label{eq_general_thm_pointwise_estimate_multiplier_conclusion}
        \Vert  \Fm_*  \Vert_{L^2(\mathbb{R}^d)} \leq C \Vert a \Vert_{L^\infty(\mathbb{R}^d)}^{ 1/4}~.
    \end{equation}
\end{thm}

 We finish the introduction with a discussion of the required variance and thin shell type bounds. As far as we know such estimates have not been previously used in the study of maximal functions. Let $X$ be a random vector in $\R^d$ distributed according to a log concave probability density $K$ which is isotropic with variance $\sigma^2.$ We call such an $X$ an isotropic log concave random vector. It is clear that  $\mathbb{E}(|X|-\sigma \sqrt{d})^2\le 4\sigma^2 d.$
Over the last twenty years non-trivial improvements of this inequality of the form
\begin{equation}
\label{eq:vb_nontr}
\mathbb{E}(|X|-\sigma\sqrt{d})^2\le  \delta_d^2\sigma^2,
\end{equation}
where $\delta_d=o(\sqrt{d})$ has been a central goal in high-dimensional convex geometry, starting with \cite{Kla0} and culminating with the recent breakthrough result of Klartag and Lehec \cite{KlaLeh}, where an optimal bound is proved. 
\begin{thm}[Variance type bound; restatement of {\cite[Corollary 1.2]{KlaLeh}}]
\label{ts_thm}
Let $X$ be an isotropic log concave random vector of variance $\sigma$. Then there exists a universal constant $C>0$ such that
\begin{equation}\label{eq:thin_shell}
\mathbb{E}(|X|-\sigma\sqrt{d})^2\le C \sigma^2.
\end{equation}
\end{thm}
\noindent The above variance type bound implies a thin-shell type bound. Namely, with very high probability, $X$ concentrates in a thin spherical shell of radius $\sigma\sqrt{d}$, cf.\ \cite[Corollary 1.3]{KlaLeh}. The variance bound from Theorem \ref{ts_thm} as well as the thin-shell type bound it implies are crucial tools in our proof of Theorem \ref{thm_av_logconc}. These are utilized in the proof of the Fourier transform estimates from Proposition \ref{propo:Linfty_diff}.
\begin{remark}
\label{rem:vb_imp}
For our main goal of proving Theorem \ref{thm_av_logconc} any power type improvement in \eqref{eq:vb_nontr}, of the form $\delta_d=O(d^{1/2-\varepsilon})$, for some $\varepsilon>0$, would suffice. However, better estimates lead to a better control of the constant $\alpha_2,$ and hence, $\alpha_p,$ $p>1.$ In particular, using \eqref{eq:thin_shell} allows one to take any $\alpha_2<1/16$ and $\alpha_p<\min(2/p,2(1-1/p))/16.$ 
\end{remark}

\subsection{Historical context and motivation} 
The study of the high-dimensional behavior of the $L^p(\mathbb{R}^d)$-operator norms of maximal functions associated with symmetric convex bodies originated in the work of Stein \cite{St82}, see also \cite{SS83}, which established dimension-free estimates for the centered Hardy-Littlewood maximal function over balls, $\mathcal{B}_\ast$. These estimates were proven as a consequence of the pointwise bound $\mathcal{B}_\ast f(x)\leq \mathcal{S}_\ast f(x)$ and of the fact that the operator norm on $L^p(\mathbb{R}^d)$ of the spherical maximal function $\mathcal{S}_\ast$ 
is non-increasing with the dimension $d$. The latter fact was established relying on Stein's method of rotations. Later, Bourgain \cite{B86} proved dimension-free estimates on $L^2(\mathbb{R}^d)$ for any maximal function $\mathcal{K}_\ast^G$ associated with a symmetric convex body $G\subseteq \mathbb{R}^d$. Interpolation with the trivial estimate $\|\mathcal{K}_\ast^G\|_{L^\infty(\mathbb{R}^d)}=1$ yields dimension-free estimates for $\mathcal{K}_\ast^G$ in the range $p\in[2,\infty]$. This result has been later extended to the range $p\in (3/2,\infty]$ independently by Bourgain \cite{B86Lp} and Carbery \cite{Car1}.  An additional proof %of this result 
can be found in the survey article \cite{BMSW21}. Similar arguments give dimension-free estimates in the range $p\in (3/2,\infty]$ for any maximal function associated with a symmetric log-concave probability density \cite{DGM18}.  Next, Müller \cite{Mul1} proved the full range $p\in (1,\infty]$ of dimension-free  $L^p(\R^d)$  estimates for maximal functions associated with symmetric convex bodies satisfying a certain geometric constraint encompassing, in particular, $\ell^q$ balls for $q\in [1,\infty)$. The most recent development concerning maximal functions associated with symmetric convex bodies in $\mathbb{R}^d$ is due to Bourgain \cite{B14}, who established dimension-free estimates for the maximal function over cubes, $\mathcal{Q}_\ast$, in the range $p\in (1,\infty]$.  A major open problem in this field is the question whether for any $p>1$ and any symmetric convex body $G$ there exists a uniform dimension-free constant $C_p$ depending only on $p$ and such that $\|\mathcal{K}_\ast^G\|_{L^p(\mathbb{R}^d)}\le C_p$.

Beyond dimension-free estimates, an even more challenging question concerns the precise value of the operator norms on $L^p(\mathbb{R}^d)$ for the above-mentioned maximal functions. While it is clear that the operator norm on $L^\infty(\mathbb{R}^d)$ of any maximal function  
associated with a normalized averaging operator is identically equal to one, to the best of our knowledge
the exact value of the operator norm on $L^p(\mathbb{R}^d)$, $p\in (1,\infty),$ remains unknown for any maximal function associated with a log-concave density in every dimension $d\geq 1$. Likewise, the precise value of $\|\mathcal{S}_\ast\|_{L^p(\mathbb{R}^d)}$ is currently unknown for any $d\geq2,\, p>\frac{d}{d-1}$.

\subsubsection{The Hardy-Littlewood maximal function $\mathcal{B}_\ast$}
%However, 
For the Hardy-Littlewood maximal functions over balls, $\mathcal{B}_\ast$, more information is available in the literature, especially when $d=1$. The best constant in the weak-type $(1,1)$ inequality for $\mB_*$ was established by Melas \cite{Me03} who showed that 
$$\Vert \mB_* \Vert_{L^1(\mathbb{R})\rightarrow L^{1,\infty}(\mathbb{R})}= \frac{11+\sqrt{61}}{12}~.$$ Moreover, Grafakos, Montgomery-Smith, and Motrunich \cite{GM-SM1} provided an explicit formula for the best constant $c_p,$ $p\in (1,\infty)$ in
\[
\|\mB_* f\|_{L^p(\R)}\le c_p \|f\|_{L^p(\R)},
\]
 when the input function $f$ is positive and convex except at one point. Additionally, the norm of the uncentered Hardy-Littlewood maximal operator, $\mB_\ast^u$, in dimension $d=1$ was obtained by Grafakos and Montgomery-Smith \cite{GM-S1}
  and it satisfies $\frac{p}{p-1}\leq \| \mB_\ast^u\|_{L^p(\mathbb{R})}\leq \frac{2p}{p-1}$.

Since obtaining explicit values for $\|\mB_*\|_{L^p(\R^d)}$ seems a difficult task, one may instead try to obtain estimates from above or below. We first discuss explicit estimates from above. In dimension one, the aforementioned result of Grafakos and Montgomery-Smith \cite{GM-S1} for the uncentered Hardy-Littlewood maximal operator combined with the pointwise estimate $ \mB_\ast f(x)\leq \mB_\ast^uf(x)$ gives
$$ \|\mB_\ast\|_{L^p(\mathbb{R})}\leq 2\frac{p}{p-1}~$$
for all $p>1$. Explicit upper bounds in arbitrary dimension have been already established by Stein and Strömberg \cite{SS83}, who showed that there exists a universal constant $C>0$ such that
\begin{equation}
\label{eq: mB StStr}
\|\mB_*\|_{L^p(\R^d)}\leq C\sqrt{d}\frac{p}{p-1}~,
\end{equation}
for all dimension $d$ and all $1<p\leq\infty$. 
In \cite{AC94}, Auscher and Carro have provided the following dimension-free improvement valid for $p\geq 2$ and  $d\ge 2$
\begin{equation}
\label{eq: mB AC}
\|\mB_*\|_{L^p(\R^d)}\leq (2+\sqrt{2})^{2/p}~.
\end{equation}
 Furthermore,
from Theorem 1.2, case $k=2$ in \cite{KWZ}, one can deduce that
\begin{equation}
\label{eq: mB KWZ}
\|\mB_*\|_{L^p(\R^d)}\leq C\left(\frac{p}{p-1}\right)^{3}\|f\|_{L^p(\R^d)}~,
\end{equation}
where $C>1$ is a universal constant. This is because for $k=2$ the factorization operator $M_k^t$ from \cite{KWZ} coincides with $\mB_t,$ see e.g.\ \cite[p.\ 427]{Ver1}.
When the input function $f$ is radially decreasing Aldaz and P\'erez-L\'azaro \cite[Corollary 2.8]{APL} established that
\begin{equation}
\label{eq: mB APL raddec}
\|\sup_{t>0}|\mB_t f|\|_{L^p(\R^d)}\leq 2\left(\frac{p}{p-1}\right)^{1/p}\|f\|_{L^p(\R^d)}~,
\end{equation}
for all $p\in (1,\infty].$ Moreover, from Men\'arguez and Soria \cite[Theorem 3]{MS1} it follows that \eqref{eq: mB APL raddec} remains true for general radial functions $f$ at the price of increasing the upfront constant from $2$ to $8.$ We note that the constants on the right-hand sides of \eqref{eq: mB StStr}, \eqref{eq: mB AC},  \eqref{eq: mB KWZ}, and \eqref{eq: mB APL raddec} above are strictly larger than $p/(p-1).$  

We observe in passing that when it comes to the maximal spherical averages one may establish the following variant of \eqref{eq: mB AC} for $d\ge 3$ and $p\ge 2$
\begin{equation*}
\|\mS_*\|_{L^p(\R^d)}\leq C^{1/p},~
\end{equation*}
where $C>0$ is a universal constant.

Explicit estimates from below are even sparser.
In dimension one, the aforementioned result of Grafakos and Montgomery-Smith \cite{GM-S1} for the uncenterd Hardy-Littlewood maximal operator combined with the pointwise estimate $\frac{1}{2}\mB_\ast^uf(x)\leq \mB_\ast f(x)$ gives
$$\max \left\lbrace 1, \frac{1}{2}\frac{p}{(p-1)}\right\rbrace \leq \|\mB_\ast\|_{L^p(\mathbb{R})}~.$$
In \cite[Remark 2.9]{APL} Aldaz and P\'erez L\'azaro proved that for all $d\geq 1$ and $p>1$
\begin{equation}
\label{eq: mB bel APL}
\|\mB_{*}\|_{L^p(\R^d)}\ge \left(1+\frac{1}{2^{dp}(p-1)}\right)^{1/p}.
\end{equation}
Furthermore in \cite[Proposition]{CG1} and \cite[Theorem~3.2]{DSS} it has been shown that, for all $p\in (1,\infty)$ it holds
\begin{equation*}
\|\mB_{*}\|_{L^p(\R^d)}\ge b_{p,d},
\end{equation*}
where $b_{p,d}:=\mB_*(|x|^{-d/p})(1,0,\ldots,0)$.

\subsubsection{The Gaussian maximal function $\mathcal{G}_\ast$}
Information about the behavior of $\|\mathcal{G}_\ast \|_{L^p(\mathbb{R}^d)}$ is provided by the theory of semigroups. In particular, it is well known that 
\begin{equation}
\label{eq: mG p/(p-1)}
\|\mG_*\|_{L^p(\R^d)}\le \frac{p}{p-1},\qquad p\in (1,\infty].
\end{equation}
What lies at the core of this inequality is the fact that $\mG_t$ is a symmetric-diffusion semigroup. Then one may apply Rota's dilation theorem together with Doob's martingale maximal inequality to conlcude \eqref{eq: mG p/(p-1)}, see e.g.\ \cite[Chapter IV, Section 4, p.\ 106]{St_topics}. We refer the reader to \cite[Equation~1.20.G*]{DGM18} for a sketch of different proof of \eqref{eq: mG p/(p-1)}. While the constant $p/(p-1)$ is optimal for Doob's maximal inequality it is unclear to us if $\|\mG_*\|_{L^p(\R^d)}=p/(p-1).$ 
Moreover, we do not know if \eqref{eq: mG p/(p-1)} holds for $\|\mB_*\|_{L^p(\R^d)}.$ 

\subsubsection{Maximal function over cubes} For the maximal function over cubes $\mathcal{Q}_\ast$, it follows from the work of Bourgain \cite{B14} that for any $p\in(1,\infty]$ there exists $C_p>0$ such that for all $d\geq 1$ it holds that
$$\|\mathcal{Q}_\ast\|_{L^p(\mathbb{R}^d)}\leq C_p~.$$
However, Bourgain's method provides essentially no estimate on the size of $C_p.$ On the other hand, the result of Aldaz \cite{Acube} established that the constant in the weak-type $(1,1)$ inequality for $Q_\ast$ tends to infinity with the dimension.
Such a constant has been made more explicit by Aubrun \cite{AuCube}, who showed that it behaves like 
$\sim_\varepsilon (\log d)^{1-\varepsilon}$ for every $\varepsilon>0$, and later by Iakolev and Strömberg \cite{IaSt}, who showed that it behaves like $\sim d^{1/4}$. 

Clearly, when $d=1$ we have $\|\mathcal{Q}_\ast\|_{L^p(\mathbb{R}^d)}=\|\mathcal{B}_\ast\|_{L^p(\mathbb{R}^d)}$ as the maximal functions coincide. When $d>1$, $\mathcal{B}_\ast$ and $\mathcal{Q}_\ast$ can be related via
$$\frac{2^d}{d^{d/2}} \frac{1}{\nu_d}\leq \frac{\mathcal{B}_\ast f}{\mathcal{Q}_\ast f} \leq \frac{2^d}{\nu_d}~, $$
 where $\nu_d$ is the volume of the unit ball in $\mathbb{R}^d$, see \cite[Exercise~2.1.3]{Gr14}. In particular, when $d>1$, it is not clear whether $\|\mathcal{Q}_\ast\|_{L^p(\mathbb{R}^d)} < \|\mathcal{B}_\ast\|_{L^p(\mathbb{R}^d)}$,  or $\|\mathcal{Q}_\ast\|_{L^p(\mathbb{R}^d)}=\|\mathcal{B}_\ast\|_{L^p(\mathbb{R}^d)}$, or $\|\mathcal{Q}_\ast\|_{L^p(\mathbb{R}^d)}>\|\mathcal{B}_\ast\|_{L^p(\mathbb{R}^d)}$.

 \subsubsection{Variance and thin-shell type bounds in high-dimensional convex geometry}
We briefly provide some context for the variance type bound \eqref{eq:thin_shell}. In the case of isotropic random vectors with a radial log-concave density, the variance type bound \eqref{eq:thin_shell} can be proven more directly by an application of the Laplace method. However, the case of general isotropic log-concave random vectors is significantly more challenging and has profound implications. A first highly non-trivial bound in the general case was proved in \cite{Kla0}, showing that \eqref{eq:vb_nontr} holds with $\delta_d \le C \sqrt{d}/\log d$. Such a bound was sufficient to establish Klartag's central limit theorem for convex sets \cite{Kla0}. The importance of \eqref{eq:thin_shell} and the related thin-shell conjecture also stems from the fact that the latter implies a positive resolution of the Bourgain slicing problem. This problem was solved by Klartag and Lehec in \cite{KlaLehSP}. Their proof relies on a bound of Guan \cite{Guan}, which was  obtained in the course of proving that $\delta_d\le C \log \log d$.
Indeed, a remarkable sequence of developments followed the seminal work \cite{Kla0}, leading both to successive improvements in the bound for $\delta_d$ and to the introduction of novel tools. For a detailed account of these developments, we refer the reader to the introduction of \cite{KlaLeh} and references therein. This line of work culminated in the recent resolution of the thin-shell conjecture by Klartag and Lehec in \cite{KlaLeh}, which establishes \eqref{eq:thin_shell}.

Somewhat surprisingly, these geometric developments have remained largely separate from the study of high dimensional maximal functions.
 We view our paper as a first attempt to bridge this gap.

\subsection{Structure of the paper and our methods} 

 Section \ref{sec: general_thm} is devoted to the proof of Theorem \ref{general_thm}. The proof is a variation of the analysis in \cite[Section~4]{BMSW21} and the reasoning is split into considering the dyadic maximal function and an $\ell^2$ sum of maximal functions over dyadic intervals. Compared to \cite{BMSW21}, our  variant of the proof has the crucial advantage that it allows us to control the operator norm of a maximal function by the operator norm of the associated Fourier multiplier operator. Results with similar reasoning appeared previously in the literature, among others in \cite[Lemma~3]{B86}, \cite[Proposition i) p.\  271]{Car1}, \cite[Theorem A]{DRdF1}, and \cite[Section 3, Corollary]{RdF1}. It is likely that using these methods one can prove a variant of Theorem \ref{general_thm}. However, we found the current formulation and proof of Theorem \ref{general_thm} most convenient for our goals. 

 In Section \ref{sec:pf_thm_av_logconc} we prove Theorem \ref{thm_av_logconc}. We apply Theorem \ref{general_thm} to the difference operators $\mA^K_t-(\mG^{\sigma})_t.$ The crucial point is to prove that the multiplier symbols $a$ and $g_{\sigma},$ corresponding to $\mA^K_1$ and $\mG_{\sigma},$ respectively, satisfy \[\|a-g_{\sigma}\|_{L^{\infty}(\R^d)}\le C d^{-\delta},\] for some universal constants $\delta>0$ and $C>0.$ This is the content of Proposition \ref{propo:Linfty_diff}. Here we use the variance bound for general log concave probability densities from Theorem \ref{ts_thm} together with a thin shell bound for Gaussian random variables from   \cite[Lemma 1]{LaMa} (or from \cite[Corollary 1.3]{KlaLeh}). The assumptions of Theorem \ref{general_thm} are verified in Lemma \ref{lemma:hyp_gen_thm}, which is a straightfoward conseqeunce of \cite[Lemma 5.10]{DGM18}. We finish Section \ref{sec:pf_thm_av_logconc} with a remark about the necessity of the averaging procedure in Theorem \ref{thm_av_logconc}, see Remark \ref{rem:fail_non_rad}. 

 Next, in Section \ref{sec:rad_loconc,extremal}, using Theorem \ref{thm_av_logconc} we deduce Theorem \ref{thm_rad_logconc} and Theorem \ref{thm:extremal}.

Section \ref{sec:sphermax} contains the proof of our result for the spherical maximal function - Theorem \ref{thm:gaussian_sphere}. The core of the proof is again a pointwise estimate for the difference of the multiplier symbols coresponding to the spherical averages and the Gaussian averages $\mG_{1/\sqrt{d}}.$ This is the content of Proposition \ref{propo:pointwise_gauss_sphere}. Having this proposition we may apply Theorem \ref{general_thm}. Its assumptions are verified in Proposition \ref{propo_estimate_for_mu} in which we need auxiliary lemmas for Bessel functions from the Appendix. We finish Section \ref{sec:sphermax} by highlighting implications of Theorems \ref{thm_rad_logconc} and \ref{thm:gaussian_sphere} for a family of maximal operators connecting $\mB_\ast$ with $\mS_\ast$ considered by Dosidis and Grafakos in \cite{DG21}.

In Section \ref{sec:lowerBound} we prove Theorem \ref{propo:GausMonotonicity}. Namely, we rely on the tensor product structure of the Gaussian kernel to show that the operator norm $\|\mG_\ast\|_{L^p(\mathbb{R}^d)}$ is non-decreasing with the dimension. Then, we compute a lower bound for $\|\mG_\ast^1\|_{L^p(\mathbb{R})}$ (and, consequently, for $\|\mG_\ast^d\|_{L^p(\mathbb{R}^d)}$) by testing $\mG_\ast^1$ on a Gaussian function. 

Additionally, in Section \ref{sec:ConcludingRemarks} we discuss the norms on $L^p(\R^d)$ of maximal functions associated with radial log-concave densities applied to three natural radial inputs: a homogenous function, the characterstic function of the ball, and a Gaussian function. From Theorems \ref{thm_rad_logconc} and \ref{thm:gaussian_sphere} it follows that we may choose operators for testing among $\mS_\ast,\,\mathcal{G}_\ast,\,\mathcal{B}_\ast.$ We observe that, as $d\to \infty,$ all of these examples give the trivial conclusion $\lm(p)\ge 1$, leading to the following question.
 \begin{Que}
	\label{que: limrad}
    Fix $p\in (1,\infty)$ and let
$\|\mG_*\|_{L^p_{rad}(\R^d)},$ $\|\mB_*\|_{L^p_{rad}(\R^d)},$ $\|\mS_*\|_{L^p_{rad}(\R^d)},$ be the operator norms of the Gaussian, ball, and spherical maximal functions on $L^p(\R^d)$ restricted to radial functions. Similarly, let $\|\mK_*\|_{L^p_{rad}(\R^d)}$ be the operator norm of a maximal function $\mK_\ast$ associated with a radial log-concave density $K\in\mathcal{L}_d^{rad},$ when restricted to radial functions.
Is it true that
\begin{equation}
 \label{eq: limrad 1}
    \lim_{d\to\infty}   \|\mG_*\|_{L^p_{rad}(\R^d)} =\lim_{d\to\infty}   \|\mB_*\|_{L^p_{rad}(\R^d)}=\lim_{d\to\infty}   \|\mS_*\|_{L^p_{rad}(\R^d)}=
     \lim_{d\to\infty} \sup_{K\in\mathcal{L}_d^{rad}}\|\mK_*\|_{L^p_{rad}(\R^d)} =
    1\quad  ?
    \end{equation}		
\end{Que}

Finally, in Appendix \ref{sec: auxlem} we record some useful identities and estimates involving Bessel functions that are needed in Section \ref{sec:sphermax}.

\subsection{Notation}
\begin{enumerate}
\item Throughout the paper the letter $d\in \N$ is reserved for the dimension and all inexplicit constants will be
independent of $d$.  
\item For two nonnegative quantities $X, Y$
we write $X \lesssim Y$ if there is an absolute universal constant
such that $X\le CY$. In particular, if the quantities $X$ and $Y$ involve the dimension $d,$ or an additional parameter such as $\xi\in \R^d,$ then the universal constant $C$ in the estimate  $X\le CY$ is independent of these parameters. For instance, for a function $a\colon\R^d\to \C$
the estimate $|a(\xi)|\lesssim 1$ means that there is a universal constant $C>0$ such that $|a(\xi)|\le C$ for all dimensions $d$ and all $\xi \in \R^d.$
 Similarly, all the constants and implicit constants in $\lesssim$ are uniform over symmetric log-concave probability densities.

\item The Fourier transform of a function $f\in L^1(\mathbb{R}^d)$ is defined by the formula
$$\widehat{f}(\xi)=\mF(f)(\xi)=\int_{\mathbb{R}^d} f(x)e^{-2\pi i x\cdot\xi} \dd x,\qquad \xi\in\R^d~.$$
\item We denote by $P_t$ the Poisson semigroup defined for $f\in L^2$ by
\begin{align}\label{defi_Poisson_semigr}
\widehat{P_tf}(\xi)=p_t(\xi)\widehat{f}(\xi)~, \qquad  p_t(\xi)=e^{-2\pi t \tfrac{|\xi|}{\sqrt{d}}}~.
\end{align}
\item 
We will also need the resolution of the identity given by the Poisson projections
\begin{align}\label{defi_Poisson_projection}
f=\sum_{n\in\mathbb{Z}}S_nf~, \quad f\in L^2(\mathbb{R}^d)~, \qquad \text{where } \; S_n=P_{2^{n-1}}-P_{2^n}~.
\end{align}
\item 
We let $J_\nu,$ $\nu>0,$ be the Bessel function of the first kind of order $\nu.$

\item We denote by $G_\sigma$ the isotropic Gaussian kernel with variance $\sigma^2$,
$$G_\sigma(x)= \frac{1}{(2\pi\sigma^2)^{d/2}}e^{-\tfrac{|x|^2}{2\sigma^2}}~, $$ 
and by lowercase $g_\sigma$ its Fourier transform,
$$g_\sigma(\xi)=\widehat{G}_\sigma(\xi)= e^{- 2\pi \sigma^2|\xi|^2}~.$$
In particular, $g_{1/\sqrt{d}}(\xi)= e^{-\tfrac{2\pi^2|\xi|^2}{d}}$.

\item For a kernel $A\in L^1(\mathbb{R}^d)$, we denote by $A_t$ the rescaling $A_t(x)=t^{-d}A(t^{-1}x)$. We write $\mathcal{A}$ for the convolution operator associated with $A$, namely $\mathcal{A}f=A\ast f$. Similarly, we denote by $\mathcal{A}_t$ the convolution operator $\mathcal{A}_tf=A_t\ast f$. The maximal function associated with the kernel $A$ is defined by $\mathcal{A}_\ast f(x)=\sup_{t>0}|A_t\ast f(x)|$. We use the lowercase letter $a$ to denote the Fourier transform of $A$, $a=\widehat{A}$. Thus, $a(t\cdot)$ is the Fourier multiplier associated with $\mathcal{A}_t$, that is
$$\mathcal{F}(\mathcal{A}_tf)(\xi)=a(t\xi)\widehat{f}(\xi)~.$$

\end{enumerate}

\section{Transference principle - Proof of Theorem \ref{general_thm}}
Since both the assumptions \eqref{eq_general_thm_pointwise_estimate_multiplier_1}, \eqref{eq_general_thm_pointwise_estimate_multiplier_gradient} and the conlusion \eqref{eq_general_thm_pointwise_estimate_multiplier_conclusion} are invariant under scaling $\xi \to \sigma^{-1}\xi$ it suffices to establish the theorem for a single $\sigma>0.$ Throughout the proof we assume that $\sigma=\frac{1}{\sqrt{d}}.$

\label{sec: general_thm}
We start by proving the following dyadic variant of Theorem \ref{general_thm}. 
\begin{proposition}\label{general_thm_dyadic}
    Let  $a=a_d$ be a bounded function on $\mathbb{R}^d$, $d\geq 1$. Assume that there exists a constant $J$ such that $a=a_d$ satisfies pointwise a.e. the inequalities in \eqref{eq_general_thm_pointwise_estimate_multiplier_1} with $\sigma=\frac{1}{\sqrt{d}}$.
    Consider the family of Fourier multiplier operators $\lbrace \Fm_t\rbrace_{t>0}$ given by
    $$\widehat{\Fm_t f}(\xi)=a(t\xi)\widehat{f}(\xi)~.$$
    Then for all $f\in L^2(\R^d)$ we have
    \begin{align*}
        \Vert \sup_{n\in\mathbb{Z}} |\Fm_{2^n}f| \Vert_{L^2(\mathbb{R}^d)} \leq 2J^{3/4} \Vert a \Vert_{L^\infty(\mathbb{R}^d)}^{1/4}\Vert f \Vert_{L^2(\mathbb{R}^d)}.
    \end{align*}
\end{proposition}

\proof Using Plancherel's identity, estimates in  \eqref{eq_general_thm_pointwise_estimate_multiplier_1}, and the following pointwise uniform estimate
$$\sum_{n\in\mathbb{Z}}\bigg(\min \bigg\lbrace \frac{2^n|\xi|}{\sqrt{d}}, \frac{\sqrt{d}}{2^n|\xi|}\bigg\rbrace\bigg)^{3/2}\le \sum_{n\in\mathbb{Z}}\min \bigg\lbrace \frac{2^n|\xi|}{\sqrt{d}}, \frac{\sqrt{d}}{2^n|\xi|}\bigg\rbrace  \leq 4~,$$
we have that
\begin{equation*}\begin{split}
    \bigg\Vert \sup_{n\in\mathbb{Z}} |\, \Fm_{2^n}f|\, \bigg\Vert_{L^2(\mathbb{R}^d)}^2 & \leq \sum_{n\in\mathbb{Z}} \Vert \Fm_{2^n}f\Vert_{L^2}^2 \\
    & = \sum_{n\in\mathbb{Z}} \int_{\mathbb{R}^d} |a(2^n\xi)|^2|\widehat{f}(\xi)|^2 \dd\xi \\
    & \leq  J^{3/2} \Vert a \Vert_{L^\infty(\mathbb{R}^d)}^{1/2} \sum_{n\in\mathbb{Z}} \int_{\mathbb{R}^d} \bigg(\min \bigg\lbrace \frac{2^n|\xi|}{\sqrt{d}}, \frac{\sqrt{d}}{2^n|\xi|}\bigg\rbrace\bigg)^{3/2}|\widehat{f}(\xi)|^2 \dd\xi \\
    & \leq 4 J^{3/2} \Vert a \Vert_{L^\infty(\mathbb{R}^d)}^{1/2} \Vert f \Vert_{L^2(\mathbb{R}^d)}^2~.
    \end{split}
\end{equation*}
\qed

To treat the non-dyadic version of the maximal operator we rely on the following decomposition
\begin{align}\label{decomposition_maximal_operator}
    \sup_{t>0} |\Fm_tf| \leq \sup_{n\in\mathbb{Z}} |\Fm_{2^n}f| + \bigg( \sum_{n\in\mathbb{Z}} \sup_{t\in [2^n,2^{n+1}]}|\Fm_t f- \Fm_{2^n}f|^2\bigg)^{1/2}~.
\end{align}
In view of Proposition \ref{general_thm_dyadic}, to complete the proof of Theorem \ref{general_thm}  it remains to estimate the second term in the right-hand-side of \eqref{decomposition_maximal_operator}.
To this end, we rely on the following Rademacher-Menshov type numerical inequality which can be found, for example, in \cite[Lemma~2.5]{MSZK20} and which holds for all $n\in\mathbb{Z}$ and for any continuous function $h:[2^{n},2^{n+1}]\rightarrow\mathbb{C}$
$$\sup_{t\in [2^n,2^{n+1}]}|h(t)-h(2^n)|\leq \sqrt{2} \sum_{\ell\in\mathbb{N}_0} \bigg( \sum_{m=0}^{2^\ell-1} |h(2^n+2^{n-\ell}(m+1))-h(2^n+2^{n-\ell}m)|^2 \bigg)^{1/2}~.$$
In view of this, and appealing to the decomposition $f=\sum_{j\in\mathbb{Z}}S_j f$ defined in \eqref{defi_Poisson_projection},  we have 
\begin{align*}
    \bigg\Vert \bigg( & \sum_{n\in\mathbb{Z}} \sup_{t\in [2^n,2^{n+1}]}|\Fm_tf  - \Fm_{2^n}f|^2\bigg)^{1/2} \bigg\Vert_{L^2(\mathbb{R}^d)} \\ 
    & \leq \sqrt{2} \sum_{\ell \geq 0} \sum_{j\in\mathbb{Z}} 
    \bigg\Vert \bigg( \sum_{n\in\mathbb{Z}} \sum_{m=0}^{2^\ell -1}\big|(\Fm_{2^n+2^{n-\ell}(m+1)}  - \Fm_{2^n+2^{n-\ell}m})S_{j+n}f \big|^2\bigg)^{1/2} \bigg\Vert_{L^2(\mathbb{R}^d)}~.
\end{align*}
We proceed by studying the norm on the right-hand-side of the last display, namely the quantity
\begin{align}\label{last_intermediate_estimate_thm_proof}
    \bigg\Vert \bigg( \sum_{n\in\mathbb{Z}} \sum_{m=0}^{2^\ell -1}  \big|(\Fm_{2^n+2^{n-\ell}(m+1)} & -  \Fm_{2^n+2^{n-\ell}m})S_{j+n}f \big|^2\bigg)^{1/2} \bigg\Vert_{L^2(\mathbb{R}^d)} ~.
\end{align}
We will obtain two estimates for \eqref{last_intermediate_estimate_thm_proof} and then we will interpolate between them. To derive the first estimate we need some preliminary bounds. First, in view of \eqref{eq_general_thm_pointwise_estimate_multiplier_1} (recall that $\sigma=1/\sqrt{d}$), we have the bound
\begin{align*}
    |a((2^n + 2^{n-\ell}(m+1))\xi)-a((2^n+2^{n-\ell}m)\xi)|\leq 4 J\min \bigg\lbrace \frac{2^n|\xi|}{\sqrt{d}}, \frac{\sqrt{d}}{2^n|\xi|} \bigg\rbrace~.
\end{align*}
Moreover, in view of the definition of Poisson semigroup in \eqref{defi_Poisson_semigr} we have 
\begin{align*}
    |(e^{-2\pi 2^{n+j}|\xi|/\sqrt{d}}-e^{-2\pi 2^{n+j-1}|\xi|/\sqrt{d}})|\leq 8\pi \min \bigg\lbrace \frac{2^{n+j}|\xi|}{\sqrt{d}}, \frac{\sqrt{d}}{2^{n+j}|\xi|} \bigg\rbrace~.
\end{align*}
Combining the above two estimates we obtain that their product is bounded by $32\pi J2^{-|j|}$.
Using this fact together with Plancherel's identity and the estimates in \eqref{eq_general_thm_pointwise_estimate_multiplier_1} we obtain
\begin{align}\begin{split}\label{estimate_in_proof_last_intermediate_1}
    \bigg\Vert & \bigg( \sum_{n\in\mathbb{Z}} \sum_{m=0}^{2^\ell -1}  \big|(\Fm_{2^n+2^{n-\ell}(m+1)} - \Fm_{2^n+2^{n-\ell}m})S_{j+n}f \big|^2\bigg)^{1/2} \bigg\Vert_{L^2(\mathbb{R}^d)} \\
    & = \bigg( \sum_{n\in\mathbb{Z}} \sum_{m=0}^{2^\ell -1} \int_{\mathbb{R}^d} |a((2^n + 2^{n-\ell}(m+1))\xi)-a((2^n+2^{n-\ell}m)\xi)|^2 |(e^{-2\pi 2^{n+j}L|\xi|}-e^{-2\pi 2^{n+j-1}L|\xi|})|^2|\widehat{f}(\xi)|^2 \dd\xi \bigg)^{1/2}\\
    & \leq 8\sqrt{\pi}\sqrt{J} 2^{-|j|/2} \bigg( \sum_{n\in\mathbb{Z}} \sum_{m=0}^{2^\ell -1} \int_{\mathbb{R}^d} |a((2^n + 2^{n-\ell}(m+1))\xi)-a((2^n+2^{n-\ell}m)\xi)||\widehat{f}(\xi)|^2 \dd\xi \bigg)^{1/2}\\
    & \leq 16 \sqrt{\pi}J2^{-|j|/2} 2^{\ell/2} \bigg(  \int_{\mathbb{R}^d} \sum_{n\in\mathbb{Z}} \min \bigg\lbrace \frac{2^n|\xi|}{\sqrt{d}}, \frac{\sqrt{d}}{2^n|\xi|} \bigg\rbrace |\widehat{f}(\xi)|^2 \dd\xi \bigg)^{1/2}\\
    & \leq 32 \sqrt{\pi}J2^{-|j|/2} 2^{\ell/2}\Vert f \Vert_{L^2(\mathbb{R}^d)}~.
    \end{split}
\end{align}
To obtain the second estimate for \eqref{last_intermediate_estimate_thm_proof} we rely on the following inequality which follows from \eqref{eq_general_thm_pointwise_estimate_multiplier_gradient}  
\begin{align}\label{estimate_using_pointwise_gradient}
    |a((2^n + 2^{n-\ell}(m+1))\xi)-a((2^n+2^{n-\ell}m)\xi)| \leq \int_{2^n+2^{n-\ell}m}^{2^n+2^{n-\ell}(m+1)} |\langle t\xi, \nabla a(t\xi)\rangle| \frac{\dd t}{t} \leq J 2^{-\ell}~.
\end{align}
Therefore, using Plancherel's identity and inequality \eqref{estimate_using_pointwise_gradient} we obtain
\begin{align}
    \begin{split}\label{estimate_in_proof_last_intermediate_2}
 \bigg\Vert \bigg( \sum_{n\in\mathbb{Z}} \sum_{m=0}^{2^\ell -1} & \big|(\Fm_{2^n+2^{n-\ell}(m+1)}-  \Fm_{2^n+2^{n-\ell}m})S_{j+n}f \big|^2\bigg)^{1/2} \bigg\Vert_{L^2(\mathbb{R}^d)} \\
    & \leq  \bigg( \sum_{n\in\mathbb{Z}} \sum_{m=0}^{2^\ell -1} \big\Vert (\Fm_{2^n+2^{n-\ell}(m+1)} -  \Fm_{2^n+2^{n-\ell}m})S_{j+n}f \big\Vert_{L^2(\mathbb{R}^d)}^2\bigg)^{1/2}  \\
    & \leq \bigg( \sum_{n\in\mathbb{Z}} 2^{\ell} J^{(2-\alpha)} 2^{-\ell(2-\alpha)} 2^\alpha\Vert a \Vert_{L^\infty}^\alpha \Vert S_{n+j}f\Vert_{L^2}^2 \bigg)^{1/2} \\
    & \leq \sqrt{2 J^{(2-\alpha)}} 2^{\ell(-1+\alpha)/2} \Vert a \Vert_{L^\infty}^{\alpha/2}\Vert f \Vert_{L^2(\mathbb{R}^d)}~,
    \end{split}
\end{align}
for any $0<\alpha<1$.
Combining \eqref{estimate_in_proof_last_intermediate_1}, \eqref{estimate_in_proof_last_intermediate_2} we obtain  
\begin{align*}
    \bigg\Vert \bigg( & \sum_{n\in\mathbb{Z}} \sup_{t\in [2^n,2^{n+1}]}|\Fm_tf  - \Fm_{2^n}f|^2\bigg)^{1/2} \bigg\Vert_{L^2(\mathbb{R}^d)} \\ 
    & \leq \sqrt{2} \sum_{\ell \geq 0} \sum_{j\in\mathbb{Z}} 
    \bigg\Vert \bigg( \sum_{n\in\mathbb{Z}} \sum_{m=0}^{2^\ell -1}\big|(\Fm_{2^n+2^{n-\ell}(m+1)}  - \Fm_{2^n+2^{n-\ell}m})S_{j+n}f \big|^2\bigg)^{1/2} \bigg\Vert_{L^2(\mathbb{R}^d)} \\
    & \leq \sqrt{2} (32 \sqrt{\pi} J)^\beta (\sqrt{2 J^{(2-\alpha)}} \Vert a \Vert_{L^\infty}^{\alpha/2} )^{1-\beta}\sum_{\ell \geq 0} \sum_{j\in\mathbb{Z}} (2^{-|j|/2} 2^{\ell/2})^\beta (2^{\ell(-1+\alpha)/2})^{1-\beta} \Vert f \Vert_{L^2(\mathbb{R}^d)}~,
\end{align*}
where $0<\beta<1$ is such that $\beta < \tfrac{1-\alpha}{2-\alpha}$, ensuring summability of the double series in the last line.
Choosing appropriate numerical values for the parameters $\alpha$ and $\beta$ (e.g. $\alpha=7/12$ and $\beta=1/7$) concludes the proof of Theorem \ref{general_thm}.
 \qed

 \begin{remark}
     The hypotheses \eqref{eq_general_thm_pointwise_estimate_multiplier_1} and \eqref{eq_general_thm_pointwise_estimate_multiplier_gradient} in the statement of Theorem \ref{general_thm} are not sharp. For example, the pointwise conditions in \eqref{eq_general_thm_pointwise_estimate_multiplier_1} can be replaced by
$$|a(\xi)|\leq C_1 \big(\sigma|\xi|\big)^b,\qquad |a(\xi)|\le C_2\bigg( \frac{1}{\sigma|\xi|}\bigg)^b~,$$
for some $b>0$. 
Moreover, Theorem \ref{general_thm} can be  
  also reformulated in the spirit of the abstract approach in \cite{MSZK20}. In particular, one may prove an $r$-variational estimate instead of the maximal function estimate. 
 \end{remark}

\section{Averages of log concave denisities - Proof of Theorem \ref{thm_av_logconc} }
\label{sec:pf_thm_av_logconc}
In this section we present the proof of Theorem \ref{thm_av_logconc} which will be obtained with the aid of Theorem \ref{general_thm}.

The core of the proof is the following pointwise estimate for the difference of multiplier symbols, which we establish relying crucially on the variance type bound \eqref{eq:thin_shell}. 
\begin{proposition}\label{propo:Linfty_diff}
    Let $K$ be a symmetric, log-concave probability density on $\mathbb{R}^d$ which is isotropic with variance $\sigma^2$. Let $A$ be the spherical average \eqref{eq:spherA} of $K$, and let $a$ denote the Fourier transform of $A$, $a:=\widehat{A}$. Then for all $\varepsilon>0$ there exists $C_\varepsilon>0$, uniform over the dimension $d\ge 1$ and the choice of $K$, such that for all $d\geq 1$
    \begin{align}\label{eq:Linfty_diff}
    |a(\xi)-g_{\sigma}(\xi)|
    \leq C_\varepsilon d^{-1/4+\varepsilon}~, \qquad \xi\in\mathbb{R}^d~.
    \end{align}
\end{proposition}
In the proof of Proposition \ref{propo:Linfty_diff} we will make use of the following identities.
\begin{lemma} 
The following pointwise identities hold for $\xi\in\mathbb{R}^d$
\begin{align}\label{eq:rewritehatA}
a(\xi):=\widehat{A}(\xi)= \int_{\R^d} K(y) \bigg(\mathop{\mathbb{E}}_{X\sim\mathcal{N}(0,I_d)} e^{-2\pi i |y|\langle \tfrac{X}{|X|},\xi \rangle} \bigg) \dd y~,
\end{align}
\begin{align}\label{eq:rewritehatg}
g_{1/\sqrt{d}}(\xi) = \mathop{\mathbb{E}}_{X\sim \mathcal{N}(0,I_d)} e^{-2\pi i \langle \tfrac{X}{\sqrt{d}}, \xi\rangle}~.
\end{align}
\end{lemma}
The proof of the lemma is a standard computation that we include for the reader's convenience.
\proof We justify \eqref{eq:rewritehatA} first. From the definition of spherical average \eqref{eq:spherA} and Fubini's theorem, we have
\begin{align*}
    a(\xi)  = \int_{\R^d} \bigg(\int_{SO(d)} K(Ry)\dd R \bigg) e^{-2\pi i \langle y,\xi \rangle}\dd y = \int_{\R^d} K(y) \bigg(\int_{SO(d)} e^{-2\pi i \langle R^{T}y,\xi\rangle}\dd R \bigg) \dd y ~,
\end{align*}
and, in view of the identities
$$\int_{SO(d)} e^{-2\pi i \langle R^{T}y,\xi\rangle}\dd R = \;\mathop{\mathbb{E}}_{U\sim \text{Unif}(\mathbb{S}^{d-1})}  e^{-2\pi i |y|\langle U,\xi\rangle} = \; \mathop{\mathbb{E}}_{X\sim \mathcal{N}(0,I_d)} e^{-2\pi i |y|\langle \tfrac{X}{|X|},\xi \rangle}~,$$
we obtain \eqref{eq:rewritehatA}. To justify \eqref{eq:rewritehatg} we use the fact that the Fourier transform of a Gaussian function is itself a Gaussian function
\begin{align*}
g_{1/\sqrt{d}}(\xi)  = e^{-2\pi^2|\xi|^2/d}
 = \frac{1}{(2\pi)^{d/2}}\int_{\R^d} e^{-2\pi i \langle \tfrac{x}{\sqrt{d}}, \xi\rangle} e^{-|x|^2/2}\dd x  = \mathop{\mathbb{E}}_{X\sim \mathcal{N}(0,I_d)} e^{-2\pi i \langle \tfrac{X}{\sqrt{d}}, \xi\rangle}~.
\end{align*}
\qed

We proceed with the proof of Proposition \ref{propo:Linfty_diff}.

\textit{Proof of Proposition \ref{propo:Linfty_diff}.} By invariance under rescaling, we can assume that $\sigma=1/\sqrt{d}$. Let $\alpha \in (0,1/2)$ be a parameter to be chosen later. From the estimates in Lemma \ref{lemma:hyp_gen_thm} display \eqref{eq:pointwise_mult1} with $\sigma=1/\sqrt{d}$, we see that for all $\xi\in\mathbb{R}^d$, $|\xi|\notin [d^{1/2-\alpha},d^{1/2+\alpha}]$, it holds 
\begin{align}\label{eq:pwdiff_ag_case1}
|a(\xi)- g_{1/\sqrt{d}}(\xi)|\lesssim \frac{1}{d^{\alpha}}~.
\end{align}
Here, the implicit constant is uniform with respect to $K$.

Thus, we focus on the case $|\xi|\in [d^{1/2-\alpha},d^{1/2+\alpha}]$. 
We use \eqref{eq:rewritehatA} and \eqref{eq:rewritehatg} to rewrite
\begin{align*}
    |a(\xi)-g_{1/\sqrt{d}}(\xi)|= \bigg|\int_{\mathbb{R}^d} K(y)\bigg(  \mathop{\mathbb{E}}_{X\sim\mathcal{N}(0,\mathrm{I}_d)} e^{-2\pi i |y|\langle \frac{X}{|X|},\xi\rangle} - e^{-2\pi i\langle \frac{X}{\sqrt{d}},\xi\rangle}\bigg) \dd y \bigg|~.
\end{align*}
Let $\frac{1}{2} > \varepsilon >0$ be a parameter to be chosen later. We split the expectation over the two events $||X|-\sqrt{d}|> d^\varepsilon$ and $||X|-\sqrt{d}|\leq d^\varepsilon$ and we apply the triangle inequality to obtain
\begin{align}\label{eq:splittingExp}\begin{split}
    \bigg|\int_{\mathbb{R}^d} K(y) &\bigg(  \mathop{\mathbb{E}}_{X\sim\mathcal{N}(0,\mathrm{I}_d)} e^{-2\pi i |y|\langle \frac{X}{|X|},\xi\rangle} - e^{-2\pi i\langle \frac{X}{\sqrt{d}},\xi\rangle}\bigg) \dd y \bigg|\\
    \leq & \, \int_{\mathbb{R}^d} K(y)\bigg( \mathop{\mathbb{E}}_{X\sim\mathcal{N}(0,\mathrm{I}_d)} \mathbf{1}_{\lbrace ||X|-\sqrt{d}|\geq d^\varepsilon\rbrace}(X)\bigg| e^{-2\pi i |y|\langle \frac{X}{|X|},\xi\rangle} - e^{-2\pi i\langle \frac{X}{\sqrt{d}},\xi\rangle}\bigg|\bigg) \dd y \\
    & + \int_{\mathbb{R}^d} K(y)\bigg(  \mathop{\mathbb{E}}_{X\sim\mathcal{N}(0,\mathrm{I}_d)} \mathbf{1}_{\lbrace ||X|-\sqrt{d}|\leq d^\varepsilon\rbrace}(X)\bigg| e^{-2\pi i |y|\langle \frac{X}{|X|},\xi\rangle} - e^{-2\pi i\langle \frac{X}{\sqrt{d}},\xi\rangle}\bigg|\bigg) \dd y~.
    \end{split}
\end{align}
To deal with the first term
we make use of the following concentration property of $d$-dimensional Gaussian random vectors which is a particular case of \cite[Corollary 1.3]{KlaLeh}: let $X\sim\mathcal{N}(0,\mathrm{I}_d)$, then for any $t>0$ we have
$$\mathbb{P}_X\bigg(\big||X|-\sqrt{d}\big|\geq t \bigg)\leq \mathbb{P}_X\bigg(\frac{\big||X|^2-{d}\big|}{\sqrt{d}}\geq t \bigg)\lesssim e^{-c\sqrt{t}}~,$$
where $c>0$ is a universal constant.
By choosing $t=d^\varepsilon$ we obtain
$$ \int_{\mathbb{R}^d} K(y)\bigg( \mathop{\mathbb{E}}_{X\sim\mathcal{N}(0,\mathrm{I}_d)} \mathbf{1}_{\lbrace ||X|-\sqrt{d}|\geq d^\varepsilon\rbrace}(X)\bigg| e^{-2\pi i |y|\langle \frac{X}{|X|},\xi\rangle} - e^{-2\pi i\langle \frac{X}{\sqrt{d}},\xi\rangle}\bigg|\bigg) \dd y \lesssim e^{-c d^{\varepsilon/2}}~.$$
It remains to bound the second term on the right-hand side of \eqref{eq:splittingExp}. For each $y$ we have
\begin{align*}
  \mathop{\mathbb{E}}_{X\sim\mathcal{N}(0,\mathrm{I}_d)} & \mathbf{1}_{\lbrace ||X|-\sqrt{d}|\leq d^\varepsilon\rbrace}(X)\bigg| e^{-2\pi i |y|\langle \frac{X}{|X|},\xi\rangle} - e^{-2\pi i\langle \frac{X}{\sqrt{d}},\xi\rangle}\bigg| \\
   & \lesssim \mathop{\mathbb{E}}_{X\sim\mathcal{N}(0,\mathrm{I}_d)}  \mathbf{1}_{\lbrace ||X|-\sqrt{d}|\leq d^\varepsilon\rbrace}(X) \bigg| \frac{|y|}{|X|}-\frac{1}{\sqrt{d}} \bigg| \big|\langle X,\xi\rangle \big| \\
   & = \mathop{\mathbb{E}}_{X\sim\mathcal{N}(0,\mathrm{I}_d)}  \mathbf{1}_{\lbrace ||X|-\sqrt{d}|\leq d^\varepsilon\rbrace}(X) \bigg| \frac{|y|-1+1}{|X|}-\frac{1}{\sqrt{d}} \bigg| \big|\langle X,\xi\rangle \big| \\
   & \leq \mathop{\mathbb{E}}_{X\sim\mathcal{N}(0,\mathrm{I}_d)}  \mathbf{1}_{\lbrace ||X|-\sqrt{d}|\leq d^\varepsilon\rbrace}(X) \bigg( \frac{||y|-1|}{|X|} \big|\langle X,\xi\rangle \big| + \bigg| \frac{1}{|X|}-\frac{1}{\sqrt{d}} \bigg| \big|\langle X,\xi\rangle \big| \bigg)~.
\end{align*}
Recall that we are in the regime $|\xi|\in [d^{1/2-\alpha}, d^{1/2+\alpha}]$. We estimate
\begin{align*}
    \int_{\mathbb{R}^d}  K(y)\dd y &\bigg(\mathop{\mathbb{E}}_{X\sim\mathcal{N}(0,\mathrm{I}_d)}  \mathbf{1}_{\lbrace ||X|-\sqrt{d}|\leq d^\varepsilon\rbrace}(X) \bigg| \frac{1}{|X|}-\frac{1}{\sqrt{d}} \bigg| \big|\langle X,\xi\rangle \big|\bigg) \\
    & = \bigg(\mathop{\mathbb{E}}_{X\sim\mathcal{N}(0,\mathrm{I}_d)}  \mathbf{1}_{\lbrace ||X|-\sqrt{d}|\leq d^\varepsilon\rbrace}(X)  \frac{||X|^2-d|}{|X|^2\sqrt{d}+|X|d}  \big|\langle X,\xi\rangle \big|\bigg) \\
    & \leq d^{1/2+\alpha} \frac{d^{1/2+\varepsilon}}{d^{3/2}-d^{1+\varepsilon}}\mathop{\mathbb{E}}_{X_1\sim\mathcal{N}(0,1)} |X_1|\\
    & \lesssim d^{-1/2+\alpha+\varepsilon}~,
\end{align*}
where in the second-last line we have used invariance under rotations of $|X|$.
It remains to bound
$$\int_{\mathbb{R}^d} K(y)\bigg(  \mathop{\mathbb{E}}_{X\sim\mathcal{N}(0,\mathrm{I}_d)} \mathbf{1}_{\lbrace ||X|-\sqrt{d}|\leq d^\varepsilon\rbrace}(X)  \frac{||y|-1|}{|X|} \big|\langle X,\xi\rangle \big| \bigg) \dd y ~.$$
We rely on the variance type bound \eqref{eq:thin_shell} with $\sigma=1/\sqrt{d}$.
By Cauchy--Schwarz, this immediately implies
$$\mathbb{E}_Y||Y|-1|\lesssim\frac{1}{\sqrt{d}}~.$$
Using this fact, together with invariance under rotations of $X$ and the fact that $|\xi|\in[d^{1/2-\alpha},d^{1/2+\alpha}]$, we obtain

\begin{align*}
  \int_{\mathbb{R}^d} K(y)& \bigg(  \mathop{\mathbb{E}}_{X\sim\mathcal{N}(0,\mathrm{I}_d)} \mathbf{1}_{\lbrace ||X|-\sqrt{d}|\leq d^\varepsilon\rbrace}(X)  \frac{||y|-1|}{|X|} \big|\langle X,\xi\rangle \big| \bigg) \dd y  \\
  & \lesssim \frac{d^{1/2+\alpha}}{d^{1/2}-d^\varepsilon}\mathop{\mathbb{E}}_{X_1\sim\mathcal{N}(0,1)} |X_1| \int_{\mathbb{R}^d} K(y)||y|-1|\dd y\\
  & \lesssim d^{-1/2+\alpha}~.
\end{align*}
Combining all together, we see that for $|\xi|\in[d^{1/2-\alpha},d^{1/2+\alpha}]$
$$|a(\xi)-g_{1/\sqrt{d}}(\xi)| \lesssim e^{-c d^{\varepsilon/2}}+d^{-1/2+\alpha+\varepsilon}+ d^{-1/2+\alpha}~.$$
In view of \eqref{eq:pwdiff_ag_case1} and the last display, choosing $\alpha=1/4$ and $\varepsilon\rightarrow 0^+$ we conclude.

\qed

To invoke Theorem \ref{general_thm}, we need to verify that the multiplier symbols satisfy the hypothesis \eqref{eq_general_thm_pointwise_estimate_multiplier_1}, \eqref{eq_general_thm_pointwise_estimate_multiplier_gradient} . This will be an immediate consequence of the following lemma.
\begin{lemma}\label{lemma:hyp_gen_thm}
    Let $K$ be a symmetric, log-concave probability density of $\mathbb{R}^d$ which we assume isotropic with variance $\sigma^2$. Let $A$ be the spherical average \eqref{eq:spherA} of $K$, and let $a$ denote the Fourier transform of $A$, $a:=\widehat{A}$. Then
    \begin{align}\label{eq:pointwise_mult1}
        |a(\xi)-1|\leq 2\pi \sigma |\xi|~, \qquad |a(\xi)|\leq \frac{1}{2\pi\sigma |\xi|}~,
    \end{align}
    \begin{align}\label{eq:pointwise_mult2}
        |\langle \xi, \nabla a(\xi)\rangle|\leq 2~.
    \end{align}
\end{lemma}
\proof
The analogous estimates for $k:=\widehat{K}$ can be found in \cite[Lemma 5.10]{DGM18}, adapting \cite[Section~4]{B86}.
By the definition of $A$ we get that
$$a(\xi)=\int_{SO(d)}k(R\xi)\dd R~.$$
Hence, the pointwise estimates in \eqref{eq:pointwise_mult1}
immediately follow from the analogous estimates for $k$. Moreover,
$$|\langle \xi , \nabla_\xi a(\xi)\rangle|=\bigg|\int_{{SO(d)}} \langle R\xi , \nabla k(R\xi)\rangle \, \dd R \, \bigg| \leq 2$$
where the last inequality follows from the analogous estimate for $k$. 

\qed

\subsection{Proof of Theorem \eqref{thm_av_logconc}}
We are ready to proceed with the proof of Theorem \eqref{thm_av_logconc}. We justify \eqref{eq:thm_av_logconc_diff_lim} first.
We apply Theorem \ref{general_thm} to the difference $(\mathcal{A}^K-\mathcal{G}_\sigma)$. It follows from Lemma \ref{lemma:hyp_gen_thm} that $\mathcal{F}(A^K- G_\sigma)$ satisfies the hypothesis \eqref{eq_general_thm_pointwise_estimate_multiplier_1}, \eqref{eq_general_thm_pointwise_estimate_multiplier_gradient} in Theorem \ref{general_thm}. By an application of Theorem \ref{general_thm} combined with the estimate \eqref{eq:Linfty_diff} we obtain
\begin{align}\label{eq:L2_diff}
    \|(\mathcal{A}^K-\mathcal{G}_\sigma)_\ast\|_{L^2(\mathbb{R}^d)} \lesssim  d^{-1/5}~.
\end{align}
On the other hand, one can easily check that
\begin{align}\label{eq:dim-freeLq_diff}
    \sup_{d\geq 1} \|(\mathcal{A}^K-\mathcal{G}_\sigma)_\ast\|_{L^q(\mathbb{R}^d)} <\infty~,
\end{align}
$q\in (1,\infty]$. In fact, dimension-free estimates for $\|\mathcal{G}_\ast\|_{L^q(\mathbb{R}^d)}$ are well-known, while dimension-free estimates for $\|\mathcal{A}^K_\ast\|_{L^q(\mathbb{R}^d)}$ immediately follow from dimension-free estimates for $\|\mathcal{S}_\ast\|_{L^q(\mathbb{R}^d)}$ in view of the pointwise bound $\mathcal{A}^K_\ast f(x)\leq \mathcal{S}_\ast f(x)$ which is a simple consequence of the fact that $A^K$ is a radial probability density. Therefore, \eqref{eq:thm_av_logconc_diff_lim} follows by Marcinkiewciz interpolation \cite[Theorem 1.3.2]{Gr14} between \eqref{eq:L2_diff} and \eqref{eq:dim-freeLq_diff}.

It remains to justify \eqref{eq:thm_av_logconc_lim}. In view of Theorem \ref{propo:GausMonotonicity}, the limit $\lim_{d\rightarrow\infty}\|\mathcal{G}_\ast\|_{L^p(\mathbb{R}^d)}=\lambda(p)$ exists. Moreover, in view of \eqref{eq:thm_av_logconc_diff_lim} we have that
$$ \big| \|\mathcal{A}^K_\ast\|_{L^p(\mathbb{R}^d)} - \|\mathcal{G}_\ast\|_{L^p(\mathbb{R}^d)} \big| \leq C_pd^{-\alpha_p}~,$$
and this is uniform with respect to $K\in\mathcal{L}_d$. This justifies \eqref{eq:thm_av_logconc_lim}.
\qed

\begin{remark} 
\label{rem:fail_non_rad}
The crucial inequality \eqref{eq:thm_av_logconc_diff_lim} from Theorem \ref{thm_av_logconc} may not hold if we consider a generic (i.\ e., not necessarily radial) symmetric log-concave density instead of its spherical average. A simple counterexample is provided by the following lemma, considering the maximal function associated with the difference of averages over symmetric cubes, 
\[\mathcal{Q}_tf(x):=t^{-d}\int_{[-t/2,t/2]^d} f(x-y)\,dy~,\]
and Gaussians. Note that $\mathbf{1}_{[-1/2,1/2]^d}$ is a log-concave isotropic density with variance $1/12,$ so the appropriate Gaussian averages  for comparison are $(\mathcal{G}^{\sqrt{1/12}})_t.$

\begin{lemma}\label{propo:lower_bound_Linfty}
For all $d\geq 1$ we have
\begin{equation*}
    \|(\mQ-\mathcal{G}^{\sqrt{1/12}})_\ast\|_{L^2(\mathbb{R}^d)}\geq e^{-\pi^2/6} ~. 
\end{equation*}
In particular, the difference does not go to zero with the dimension $d\to \infty.$
\end{lemma}
\proof Clearly,
$$\|(\mQ-\mathcal{G}^{\sqrt{1/12}})_\ast\|_{L^2(\mathbb{R}^d)} \geq \|(\mQ-\mathcal{G}^{\sqrt{1/12}})_{t=1}\|_{L^2(\mathbb{R}^d)}~.$$
Hence, by Plancherel's theorem, the statement is a simple consequence of the fact that
\begin{align*}
    \sup_{\xi\in\mathbb{R}^d} \big| |Q_d|^{-1}\widehat{\mathbf{1}_{Q_d}}(\xi) - g_{\sqrt{1/12}}(\xi)\big| & = \sup_{\substack{\xi\in\mathbb{R}^d\\ \xi=(\xi_1,...,\xi_d)}} \bigg| \prod_{j=1}^d \frac{\sin(\pi \xi_j)}{\pi \xi_j} - e^{-\pi^2|\xi|^2/6}\bigg|\\
    & \geq \sup_{\xi\in\mathbb{R}} \bigg| \frac{\sin(\pi \xi)}{\pi \xi} - e^{-\pi^2\xi^2/6}\bigg|\\
    & \ge\, e^{-\pi^2/6}~.
\end{align*}
\qed
\end{remark}

\begin{remark}
Log-concavity may not be preserved by spherical averaging. A simple counterexample is provided below. In particular, the family of radial log-concave densities considered in Theorem \ref{thm_rad_logconc} is strictly contained in the family of spherical averages of symmetric log-concave densities considered in Theorem \ref{thm_av_logconc}.    
\end{remark}
\begin{example}
    Consider the square $Q=[-1/2,1/2]^2\subset\mathbb{R}^2$. The spherical average \eqref{eq:spherA} of $|Q|^{-1}\mathbf{1}_{Q}$ is
    \begin{align*}
        A(x_1,x_2)=\frac{1}{2\pi}\int_0^{2\pi}\mathbf{1}_Q(|x|\cos\theta,|x|\sin\theta)\dd \theta = \begin{cases}
1 & 0\leq |x| \leq 1/2\\
1 -\frac{4}{\pi}\arccos{\frac{1}{2|x|}} & 1/2<|x|\leq 1/\sqrt{2}\\
0 & |x|>1/\sqrt{2}
\end{cases} ~.
    \end{align*}
We recall that a function $f$ on $\R^d$ is log-concave if it satisfies the condition $f((1-t)x+ty)\geq f(x)^{1-t}f(y)^t$, for all $x,y\in \mathbb{R}^d$, $t\in[0,1]$. We check this condition for $A$ at the points $x=(0.51,0)$, $y=(0.58,0)$, with $t=4/7$, obtaining
\begin{align*}
&\text{LHS}:= A((\tfrac{3}{7}0.51+\tfrac{4}{7}0.58,0)) \approx 0.4528893927 ~, \\ 
&\text{RHS}:= A((0.51,0))^{3/7}A((0.58,0))^{4/7}\approx 0.4630364024 ~.
\end{align*}
In particular, $\text{LHS}<\text{RHS}$ and $A$ is not log-concave on its whole domain.
\end{example}

\section{Proofs of Theorems \ref{thm_rad_logconc} and \ref{thm:extremal}}
\label{sec:rad_loconc,extremal}

\subsection{Proof of Theorem \ref{thm_rad_logconc}}
It is clear that \eqref{eq:thm_rad_logconc_diff_lim} is a particular case of \eqref{eq:thm_av_logconc_diff_lim}.

Existence of the limit $\lim_{d\rightarrow\infty}\|\mathcal{G}_\ast\|_{L^p}=:\lm(p)$ is a consequence of the fact that the operator norm of the Gaussian maximal function, $\|\mathcal{G}_\ast\|_{L^p(\mathbb{R}^d)}$, is non-decreasing with the dimension $d$. This fact is established in Theorem \ref{propo:GausMonotonicity}. To justify \eqref{eq:thm_rad_logconc_lim}, we further observe that as a consequence of \eqref{eq:thm_rad_logconc_diff_lim}
$$\big| \|\mathcal{K}_\ast\|_{L^p(\mathbb{R}^d)} -\|\mathcal{G}_\ast\|_{L^p(\mathbb{R}^d)} \big| \leq C_p d^{-\alpha_p}~$$
and this is uniform with respect to $K\in\mathcal{L}_d^{rad}$.

Finally, the lower bound for $\lm(p)$ in \eqref{eq: lm bounds} is an immediate consequence of Theorem \ref{propo:GausMonotonicity} while the upper-bound for $\lm(p)$ in \eqref{eq: lm bounds} is a consequence of the known bound \eqref{eq: mG p/(p-1)}.
\qed

\subsection{Proof of Theorem \ref{thm:extremal}}
As a consequence of \eqref{eq:thm_av_logconc_lim} we have
$$ \lim_{d\rightarrow\infty} \sup_{K\in{\LC}} \|\mathcal{A}^K_\ast\|_{L^p(\mathbb{R}^d)} = \lim_{d\rightarrow\infty} \inf_{K\in{\LC}} \|\mathcal{A}^K_\ast\|_{L^p(\mathbb{R}^d)} = \lm(p)~. $$
Combining this with the fact that for all $K\in\mathcal{L}_d$ and all $d\geq 1$
$$\|\mathcal{A}^K_\ast\|_{L^p(\mathbb{R}^d)}\leq \|\mathcal{K}_\ast\|_{L^p(\mathbb{R}^d)}$$
we have
$$\lm(p)= \lim_{d\rightarrow\infty} \inf_{K\in{\LC}} \|\mathcal{A}^K_\ast\|_{L^p(\mathbb{R}^d)}=\liminf_{d\rightarrow\infty} \inf_{K\in{\LC}} \|\mathcal{A}^K_\ast\|_{L^p(\mathbb{R}^d)} \leq \liminf_{d\rightarrow\infty} \inf_{K\in{\LC}} \|\mathcal{K}_\ast\|_{L^p(\mathbb{R}^d)}~,$$
hence justifying \eqref{eq:liminf_any_symmlogconcave}.  \qed

\begin{remark} As a straightforward consequence of Theorem \ref{thm:extremal} we obtain a lower bound for the high-dimensional limit of the $L^p$ norms of the maximal function associated with symmetric cubes, $\mathcal{Q}_\ast$. Namely, for $p\in (1,\infty)$ 
we have
\begin{align}\label{eq:cubes_geq_lambdap}
\lambda(p)\leq \lim_{d\rightarrow \infty} \|\mathcal{Q}_\ast\|_{L^p(\mathbb{R}^d)}~.
\end{align}
\noindent The existence of the limit in \eqref{eq:cubes_geq_lambdap} is guaranteed by Remark \ref{rem: TensMax}. While, in general, it is unclear whether  $\|\mathcal{B}_\ast\|_{L^p(\mathbb{R}^d)} \leq \|\mathcal{Q}_\ast\|_{L^p(\mathbb{R}^d)} $ or   $\|\mathcal{B}_\ast\|_{L^p(\mathbb{R}^d)} \geq \|\mathcal{Q}_\ast\|_{L^p(\mathbb{R}^d)} $ (except for $d=1$ when they coincide) inequality \eqref{eq:cubes_geq_lambdap} tells us that in the limit as $d\rightarrow\infty$ we have  $\|\mathcal{B}_\ast\|_{L^p(\mathbb{R}^d)} \leq \|\mathcal{Q}_\ast\|_{L^p(\mathbb{R}^d)} $.
\end{remark}

\section{Spherical maximal function - Proof of Theorem \ref{thm:gaussian_sphere} }
\label{sec:sphermax}
In this section, we prove Theorem \ref{thm:gaussian_sphere} by applying Theorem \ref{general_thm} to the difference maximal function $(\mathcal{S}-\mathcal{G}_{1/\sqrt{d}})_\ast$. Throughout this section, we denote by $\mu$ the Fourier multiplier symbol of $\mathcal{S}$,
$$\mu(\xi)  =\frac{1}{\sigma(\mathbb{S}^{d-1})} \int_{\mathbb{S}^{d-1}} e^{- 2\pi i x\cdot \xi} \dd\sigma(x)= \frac{\Gamma(\tfrac{d}{2})}{\pi^{d/2-1}} |\xi|^{-\tfrac{d}{2}+1} J_{\tfrac{d}{2}-1}(2\pi |\xi|)~.$$

As in the proof of Theorem \ref{thm_av_logconc}, the core of the proof is a pointwise estimate on the difference of the multiplier symbols, which rely on the thin-shell property of Gaussian random vectors.
\begin{proposition}\label{propo:pointwise_gauss_sphere}
    For each $\varepsilon>0$ and all $d\geq 3$ it holds
    \begin{align}\label{eq: g-mu}
    |\mu(\xi)-g_{1/\sqrt{d}}(\xi)|\leq C_\varepsilon d^{-1/4+\varepsilon}~, \qquad \xi\in\mathbb{R}^d~,
    \end{align}
    where the constant $C_\varepsilon$ depends only on $\varepsilon$.
\end{proposition}
To prove the proposition, we make use of the following identity for $\mu$.
\begin{lemma}
    \label{lem: identity_for_mu,g} The following pointwise identity holds for $\xi \in \R^d$
\begin{align}
\label{lemma_identity_for_mu}
   \mu(\xi) = \mathop{\mathbb{E}}_{X\sim \mathcal{N}(0,\textrm{I}_d)} e^{-2\pi i \langle \tfrac{X}{|X|}, \xi\rangle}~.
\end{align}
\end{lemma}
The proof of the lemma is a standard computation that we include for the reader's convenience.
\begin{proof} We compute
\begin{align*} \mathop{\mathbb{E}}_{X\sim \mathcal{N}(0,\textrm{I}_d)} e^{-2\pi i \langle\tfrac{X}{|X|},\xi\rangle} & = \frac{1}{(2\pi)^{d/2}}\int_{\mathbb{R}^d}  e^{-2\pi i \langle \tfrac{x}{|x|},\xi\rangle} e^{-\tfrac{|x|^2}{2}}\dd x \\
& = \frac{1}{(2\pi)^{d/2}}\int_0^\infty \bigg( \int_{\mathbb{S}^{d-1}}  e^{-2\pi i\langle \omega, \xi\rangle} \dd\sigma(\omega)\bigg) e^{-\tfrac{r^2}{2}} r^{d-1}\dd r  \\
&= \sigma(\mathbb{S}^{d-1}) \mu(\xi) \frac{1}{(2\pi)^{d/2}}\int_0^\infty e^{-\tfrac{r^2}{2}} r^{d-1}\dd r \\
&= \mu(\xi) \frac{1}{(2\pi)^{d/2}}\int_{\mathbb{R}^d} e^{-\tfrac{|x|^2}{2}}\dd x \\
&= \mu(\xi)~.
\end{align*}
\end{proof}
We proceed now with the proof of Proposition \ref{propo:pointwise_gauss_sphere}.\\
\textit{Proof of Proposition \ref{propo:pointwise_gauss_sphere}.}
We split the analysis into two cases. First, we consider $\xi\in\mathbb{R}^d$, $|\xi|\notin [d^{1/2-\beta},d^{1/2+\beta}]$ for some $\beta>0$ to be chosen later. In view of the pointwise estimates in  \eqref{eq:estimates_for_mu} and in \eqref{eq:pointwise_mult1} with $\sigma=1/\sqrt{d}$, we have 
\begin{equation}
\label{eq: mu-g small-large xi}
|\mu(\xi)-g_{1/\sqrt{d}}(\xi)|\lesssim  d^{-\beta}~.\end{equation} 

Next, we consider the complementary case $|\xi| \in [d^{\tfrac{1}{2}-\beta}, d^{\tfrac{1}{2}+\beta}]$. In view of \eqref{lemma_identity_for_mu} and \eqref{eq:rewritehatg}, the following identity holds 
\begin{align*}
    |\mu(\xi)-g_{1/\sqrt{d}}(\xi)|= \bigg|\mathop{\mathbb{E}}_{X\sim \mathcal{N}(0,\textrm{I}_d)}\bigg(e^{-2\pi i \langle\tfrac{X}{|X|}, \xi\rangle} - e^{- 2\pi i \langle\tfrac{X}{\sqrt{d}}, \xi\rangle} \bigg) \bigg|~.
\end{align*}
We further split the expectation above over the two complementary events $||X|^2-d|\leq 4d^{1/2+\varepsilon}$ and $||X|^2-d| > 4d^{1/2+\varepsilon}$, for some $\varepsilon\in(0,1/2)$ to be chosen later. To deal with the second case, we rely on the following thin-shell concentration inequality 
\begin{equation}
\label{eq: concentration}
\mathbb{P}(|X|^2\in [d-2d^{1/2+\varepsilon}, d+ 2d^{1/2+\varepsilon} + 2d^{\varepsilon}])\geq 1- e^{-d^\varepsilon}.\end{equation}
This inequality is a consequence of Lemma 1 and the comment following it in \cite{LaMa} with $a_i=1,$ $i=1,\ldots,D,$ and $x=d^{\varepsilon}$. Inequality \eqref{eq: concentration} can be also seen as a variant of the thin-shell bound  \cite[Corollary 1.3]{KlaLeh} for the case of Gaussian random vectors. Using \eqref{eq: concentration} we obtain
$$\bigg|\mathop{\mathbb{E}}_{X\sim \mathcal{N}(0,\textrm{I}_d)} \mathbf{1}_{\lbrace ||X|^2-d|\, \geq \, 4d^{1/2+\varepsilon} \rbrace}\big(e^{-2\pi i \langle\tfrac{X}{|X|}, \xi\rangle} - e^{- 2\pi i\langle\tfrac{X}{\sqrt{d}}, \xi\rangle} \big) \bigg| \leq 2 e^{-d^\varepsilon} ~.$$
Hence, we are left to bound 
\begin{align*}
    \bigg|\mathop{\mathbb{E}}_{X\sim \mathcal{N}(0,\textrm{I}_d)} & \mathbf{1}_{\lbrace ||X|^2-d|\, \leq \, 4d^{1/2+\varepsilon} \rbrace}\big(e^{-2\pi i \langle\tfrac{X}{|X|}, \xi\rangle} - e^{- 2\pi i\langle \tfrac{X}{\sqrt{d}}, \xi\rangle} \big) \bigg| \\
    & \lesssim \mathop{\mathbb{E}}_{X\sim \mathcal{N}(0,\textrm{I}_d)} \mathbf{1}_{\lbrace ||X|^2-d|\, \leq \, 4d^{1/2+\varepsilon} \rbrace} \big|\tfrac{1}{|X|}-\tfrac{1}{\sqrt{d}}\big| |\langle X,\xi\rangle| \\
    &= \mathop{\mathbb{E}}_{X\sim \mathcal{N}(0,\textrm{I}_d)} \mathbf{1}_{\lbrace ||X|^2-d|\, \leq \, 4d^{1/2+\varepsilon} \rbrace} \tfrac{||X|^2-d|}{|X|^2\sqrt{d}+|X|d} |\langle X,\xi\rangle| \\
    & \lesssim \mathop{\mathbb{E}}_{X\sim \mathcal{N}(0,\textrm{I}_d)}\tfrac{d^{1/2+\varepsilon}}{d^{3/2}} |\langle X,\xi\rangle| \\
    & =d^{-1+\varepsilon}|\xi| \mathop{\mathbb{E}}_{{X_1} \sim \mathcal{N}(0,1)} |{X_1}|\\
    & \lesssim d^{-1/2+\varepsilon+\beta}
\end{align*}
 where we have used invariance under rotation of $|X|$ to pass to the second equality and the fact that, by hypothesis, $|\xi| \leq d^{\tfrac{1}{2}+\beta}$ to pass to the last line. Therefore, we justified that in the case $|\xi|\in [d^{1/2-\beta},d^{1/2+\beta}]$ it holds
 \begin{equation*}
|\mu(\xi)-g_{1/\sqrt{d}}(\xi)|\lesssim e^{-d^{\varepsilon}}+d^{-1/2+\varepsilon+\beta}~.\end{equation*} 
In view of the above inequality and \eqref{eq: mu-g small-large xi}, choosing $\beta=1/4$ and $\varepsilon\to 0^+$ we obtain \eqref{eq: g-mu}.

\qed

In order to apply Theorem \ref{general_thm}, we need to verify that the multiplier $(\mu-g)$ satisfies the hypothesis \eqref{eq_general_thm_pointwise_estimate_multiplier_1}, \eqref{eq_general_thm_pointwise_estimate_multiplier_gradient} . This will be an immediate consequence of the following pointwise estimates for $\mu$, combined with the analogous estimates for $g$ that follow, for example, from Lemma \ref{lemma:hyp_gen_thm} by choosing $\sigma=1/\sqrt{d}$. 

\begin{proposition}\label{propo_estimate_for_mu}
For all $d\geq 3$, the multiplier  $\mu$ satisfies the estimates
\begin{equation}\label{eq:estimates_for_mu}
    |\mu(\xi)-1|\lesssim \frac{|\xi|}{\sqrt{d}}~, \qquad |\mu(\xi)|\lesssim \frac{\sqrt{d}}{|\xi|},\qquad
 |\langle \xi, \nabla \mu(\xi) \rangle | \lesssim 1,
\end{equation}
for a.e.\ $\xi\in \R^d.$
\end{proposition}

The proof is very close to the analysis in \cite[Section 4]{MSW24}. However, since the uniform estimate for the gradient $ |\langle \xi, \nabla \mu(\xi) \rangle | \lesssim 1$ is only mentioned in \cite{MSW24}  without proof, see \cite[eq.\ (5.20)]{MSW24}, we decided to provide a detailed proof.

\proof
We start by showing that
$$|\mu(\xi)-1|\leq 2\pi^2\frac{|\xi|}{\sqrt{d}}~.$$
We split the analysis into two cases. First, we consider the case $|\xi|\geq \sqrt{d}$. As $|\mu(\xi)|\leq 1$, we easily have that
$$|\mu(\xi)-1|\leq 2  \frac{|\xi|}{\sqrt{d}}~.$$
Let's now consider the case $|\xi|<\sqrt{d}$. It has been shown in \cite[Equation~4.9]{MSW24} that $|\mu(\xi)-1|\leq 2\pi^2 \big( \frac{|\xi|}{\sqrt{d}}\big)^{2}$. Combining this with the fact that, by hypothesis, $|\xi|<\sqrt{d}$ we obtain that
$$|\mu(\xi)-1|\leq  2\pi^2 \frac{|\xi|}{\sqrt{d}}~.$$

Next, we want to show that 
$$|\mu(\xi)| \lesssim \frac{\sqrt{d}}{|\xi|}.$$ 
We split the analysis into three cases. Assume first that $\sqrt{d}\geq |\xi|$. It has been shown in \cite[Equation~4.10]{MSW24} that  $|\mu(\xi)|\lesssim (\sqrt{d}/|\xi|)^{1/2}$. 
As by hypothesis $\sqrt{d}/|\xi|\geq 1$ it immediately follows that in this regime
$$|\mu(\xi)|\lesssim\frac{\sqrt{d}}{|\xi|}~.$$
Next we consider the case $\sqrt{d} < |\xi|<d$. It has been shown in \cite[Lemma~4.1]{MSW24} that there exists a constant $\alpha>0$ independent of the dimension such that for all $d\geq 2$ and all $\xi\in\mathbb{R}^d$ we have
$$|\mu(\xi)|\lesssim  e^{-2\pi|\xi|/\sqrt{d}}+e^{-\alpha d}.$$
Using this estimate, the fact that $e^{-|x|}\leq \tfrac{1}{|x|}$, and the hypothesis $1 < d/|\xi|$ we have 
\begin{align*}
    |\mu(\xi)|& \lesssim e^{-2\pi|\xi|/\sqrt{d}}+e^{-\alpha d}  \\
    & \lesssim \frac{\sqrt{d}}{|\xi|} +\frac{1}{\alpha\sqrt{d}}  \lesssim   \frac{\sqrt{d}}{|\xi|} +\frac{1}{\alpha\sqrt{d}}\frac{d}{|\xi|} \\
    & \lesssim \frac{\sqrt{d}}{|\xi|}~.
\end{align*}
We are left to consider the case $|\xi|\geq d$. Here we use the pointwise bound 
\begin{align}\label{eq_bessel_leq_r_power_minus_one_half}
|J_\nu(r)|\leq r^{-1/2}
\end{align}
which holds for $\nu\geq \tfrac{1}{2}$ and $r\geq 2\nu$, see e.g. \cite[Theorem~3]{Kr14}. This shows that for $d\ge 3$ we have
\begin{align*}
    |\mu(\xi)| & =\frac{\Gamma(\tfrac{d}{2})}{\pi^{d/2-1}}|\xi|^{-\tfrac{d}{2}+1} |J_{\tfrac{d}{2}-1}(2\pi|\xi|)| \leq \frac{\Gamma(\tfrac{d}{2})}{\pi^{d/2-1}}|\xi|^{-\tfrac{d}{2}+\tfrac{1}{2}} \\
    & \leq \frac{\Gamma(\tfrac{d}{2})}{\pi^{d/2-1}}d^{-\tfrac{d}{2}+\tfrac{3}{2}}|\xi|^{-1} \lesssim\frac{1}{|\xi|}\\
    & \lesssim\frac{\sqrt{d}}{|\xi|},
\end{align*}
where in the third inequality we used Stirling's approximation for the Gamma function 
\begin{equation}
\label{eq: Stirling}
\Gamma(s) = \sqrt{2 \pi /s} (s / e)^{s} (1 + O(s^{-1})),\qquad s \to \infty.
\end{equation}

Finally, we are left to show that $|\langle \xi, \nabla \mu(\xi) \rangle|\lesssim 1$ for all $\xi\in\mathbb{R}^d$. 
We split the analysis into two cases. First, we consider the case $r=|\xi|\geq d$. Using identity \eqref{first_identity_rdmudr} together with \eqref{eq_bessel_leq_r_power_minus_one_half} and \eqref{eq: Stirling} we have that for $d\geq 3$ 
\begin{align*}
   \bigg| r\frac{\dd}{\dd r} \mu(r)\bigg|& \leq 2 \frac{\Gamma(\tfrac{d}{2})}{\pi^{d/2-1}}\big(\tfrac{d}{2}-1 \big) r^{-\tfrac{d}{2}+1}\big| J_{\tfrac{d}{2}-1}(2\pi r)\big| + 2\pi \frac{\Gamma(\tfrac{d}{2})}{\pi^{d/2-1}}r^{-\tfrac{d}{2}+2}\big|J_{\tfrac{d}{2}-2}(2\pi r)\big|\\
   & \leq 2 \frac{\Gamma(\tfrac{d}{2})}{\pi^{d/2-1}}\big(\tfrac{d}{2}-1 \big) r^{-\tfrac{d}{2}+\tfrac{1}{2}} + 2\pi \frac{\Gamma(\tfrac{d}{2})}{\pi^{d/2-1}}r^{-\tfrac{d}{2}+\tfrac{3}{2}}\\
   & \leq 2 \frac{\Gamma(\tfrac{d}{2})}{\pi^{d/2-1}}\big(\tfrac{d}{2}-1 \big) d^{-\tfrac{d}{2}+\tfrac{1}{2}} + 2\pi \frac{\Gamma(\tfrac{d}{2})}{\pi^{d/2-1}}d^{-\tfrac{d}{2}+\tfrac{3}{2}}\\
   & \lesssim 1~.
\end{align*}
Next, we consider the case $r=|\xi|\leq d$. This can be verified as outlined in \cite[Remark~5.1]{MSW24}, we include the arguments for completeness. Using identity \eqref{second_identity_rdmudr} together with the estimate \eqref{estimate_as_in_MSW} we obtain 
\begin{align*}
    \bigg| r\frac{\dd}{\dd r} \mu(r)\bigg|& \leq r (2\sqrt{\pi}) \frac{\Gamma(\tfrac{d}{2})}{\Gamma(\tfrac{d}{2}-\tfrac{1}{2})}  \bigg| \int_{-1}^1 e^{i2\pi rs}s(1-s^2)^{(d-3)/2}\dd s \bigg| \\
    & \lesssim  r \frac{\Gamma(\tfrac{d}{2})}{\Gamma(\tfrac{d}{2}-\tfrac{1}{2})}  \bigg(\frac{e^{-2\pi r/\sqrt{d}}}{d}+\frac{e^{-cd}}{\sqrt{d}}\bigg) \\
    & \lesssim r \bigg(\frac{e^{-2\pi r/\sqrt{d}}}{\sqrt{d}}+e^{-d/10}\bigg)\lesssim 1~.
\end{align*}
\qed

\subsection{Proof of Theorem \ref{thm:gaussian_sphere}}
We justify \eqref{eq:thm_spher_diff_lim} first. By an application of Theorem \ref{general_thm} together with Proposition \ref{propo:pointwise_gauss_sphere} we see that for all $d\geq 3$
\begin{align}\label{eq:L2_SminusG}
    \|(\mathcal{S}-\mathcal{G}^{1/\sqrt{d}})\|_{L^2(\mathbb{R}^d)}\lesssim d^{-1/5}~.
\end{align}
Dimension-free estimates for $\mathcal{S}_\ast,\mathcal{G}_\ast$ easily imply that for $q\in (1,\infty)$, $d_0:=\lfloor q/(q-1) \rfloor +1$
\begin{align}\label{eq:Lq_SminusG}
   \sup_{d\geq d_0} \|(\mathcal{S}-\mathcal{G}^{1/\sqrt{d}})\|_{L^q(\mathbb{R}^d)} <\infty~.
\end{align}
Marcinkiewicz interpolation \cite[Theorem 1.3.2]{Gr14} applied to \eqref{eq:L2_SminusG}, \eqref{eq:Lq_SminusG} implies \eqref{eq:thm_spher_diff_lim}.

It remains to justify \eqref{eq:limitS}. It is well known that for each $p>d/(d-1)$ we have $\|\mS_*\|_{L^p(\R^{d+1})}\le \|\mS_*\|_{L^p(\R^{d})},$ see e.g.\ \cite[eq.\ (2.10)]{BMSW21}. In particular, for each $p\in (1,\infty)$ the limit $\lim_{d\to \infty}\|\mS_*\|_{L^p(\R^d)}$ 
exists. Thus, in view of \eqref{eq:thm_spher_diff_lim}
$$\lim_{d\to \infty}\|\mS_*\|_{L^p(\R^d)}=\lim_{d\to \infty}\|\mG_*\|_{L^p(\R^d)} $$
and $\lim_{d\to \infty}\|\mG_*\|_{L^p(\R^d)}=\lm(p)$ by our definition of $\lambda(p)$. This justifies  \eqref{eq:limitS}.

\qed

\begin{remark}
\label{rem: DG}
    A family of $m-$parameters maximal operators $\mathcal{S}^m_{\alpha,\ast}$, with $0\leq \alpha <1$, connecting the ($m-$parameters) Hardy-Littlewood maximal function with the ($m-$parameters) spherical maximal function has been studied by Dosidis and Grafakos in \cite{DG21}. Here, we focus on the one-parameter case $\mathcal{S}_{\alpha,\ast}=\mathcal{S}^1_{\alpha,\ast}$, 
    $$\mathcal{S}_{\alpha,\ast}(f)(x):=\sup_{t>0} \frac{2}{\omega_{d-1}\mathrm{B}(\frac{d}{2},1-\alpha)}\int_{B}|f(x-ty)|(1-|y|)^{-\alpha}\dd y~,$$
    where $\omega_{d-1}$ is the surface area of the unit sphere in $\mathbb{R}^d$, and $\mathrm{B}(a,b)$ is the beta function $\mathrm{B}(a,b):=\int_0^1 t^{a-1}(1-t)^{b-1}\dd t$. It has been shown in \cite{DG21} that, for any $f\in L^1_{loc}(\mathbb{R}^d)$ and $x\in\mathbb{R}^d$, the pointwise estimate
    $$\mB_\ast f(x)\leq S_{\alpha,\ast}f(x)\leq \mS_\ast f(x)$$
    holds.  In view of this fact and as a consequence of our Theorems \ref{thm_rad_logconc} and \ref{thm:gaussian_sphere}, we immediately obtain that also for this family of maximal operators
    $$\lim_{d\rightarrow\infty} \|\mathcal{S}_{\alpha,\ast}\|_{L^p(\mathbb{R}^d)} = \lm(p)~.$$
    
\end{remark}

 \section{Lower bounds for $\lm(p)$ - Proof of Theorem \ref{propo:GausMonotonicity}}\label{sec:lowerBound}
We split the proof of Theorem \ref{propo:GausMonotonicity} into two parts. First, we focus on the monotonicity property for $\|\mG_\ast\|_{L^p(\mathbb{R}^d)}$. 

\begin{proposition}\label{propo2:GausMonotonicity}
    For all $p\in(1,\infty)$ and $d\geq 1$ it holds that
    $$\|\mathcal{G}_\ast\|_{L^p(\mathbb{R}^{d+1})}\geq \| \mathcal{G}_\ast\|_{L^p(\mathbb{R}^d)}~.$$
\end{proposition}
\proof We start by observing that, by the monotone convergence theorem, it suffices to consider $0<t<T$ for some fixed $T\ge 1$. 
Let $f\in L^p(\mathbb{R}^d)$, $\varphi\in L^1_{loc}$. Without loss of generality, we assume $f$ to be non-negative. We have that
\begin{equation*}
    \mG^{d+1}_t(f\otimes\varphi)(x,s)=\mG_t^df(x)\,\mG_t^1\varphi (s)~,
\end{equation*}
where $(x,s)\in\mathbb{R}^{d+1}$, $x\in\mathbb{R}^d$, $s\in\mathbb{R}$. We set \begin{equation}
\label{eq: phi,phi_N}
\varphi(s)=|s|^{-1/p}\qquad \textrm{and}\qquad  \varphi_N(s):=\varphi(s)\mathbf{1}_{[1,N]}(|s|).
\end{equation}Here $N>2$ is a large parameter which will be taken to infinity in the proof. Note that $\varphi_N, \, \varphi_{\frac{N}{2}}\in L^p(\mathbb{R})$.
Since the kernel of $\mG_t^1$ is even a change of variable gives
\begin{align}
\begin{split}\label{eq:convexity_low_b}
\mG_t^1\varphi_N (s) & = \frac{1}{\sqrt{2\pi t^2}} \int_{-\infty}^\infty \varphi_N(s-y) e^{-y^2/(2t^2)}\dd y\\
& = \frac{1}{\sqrt{2\pi }} \int_0^\infty \left( \varphi_N(s-{t}y)+\varphi_N(s+{t}y)\right) e^{-y^2/2}\dd y\\
& \geq \frac{2}{\sqrt{2\pi }} \int_0^M  \frac{\varphi_N(s-{t}y)+\varphi_N(s+{t}y)}{2} e^{-y^2/2}\dd y~,
\end{split}
\end{align}
for any $M\ge 1$. At this point, we observe that $\varphi|_{(-\infty,0)}$ and $\varphi|_{(0,\infty)}$ are convex. Moreover, if $|s|>{t}y$ then $\text{sign}(s+{t}y)=\text{sign}(s-{t}y)$. In particular, this is guaranteed for $|s|>10{T}M$. Observe also that if $10{T}M<|s|<\tfrac{N}{2}$ then both $|s+{t}y|$ and $|s-{t}y|$ lie between $1$ and $N$. Obviously, here we take $N>20{T}M.$ Therefore, in such a range for $s$ using the convexity of $\varphi$ we obtain
$$\frac{\varphi_N(s-{t}y)+\varphi_N(s+{t}y)}{2}=\frac{\varphi(s-{t}y)+\varphi(s+{t}y)}{2}\geq \varphi(s)~. $$
Plugging this into \eqref{eq:convexity_low_b} we have that for all $0<t<T$
$$\mG_t^1\varphi_N (s) \geq \varphi(s)\mathbf{1}_{[ 10TM,N/2]}(|s|) \frac{2}{\sqrt{2\pi}}\int_0^Me^{-y^2/2}\dd y\geq \varphi_{\frac{N}{2}}(s)\mathbf{1}_{[ 10{T}M,\infty)}(|s|) I(M)~, $$
where $I(M):= \frac{2}{\sqrt{2\pi}}\int_0^Me^{-y^2/2}\dd y$ satisfies $\lim_{M\rightarrow\infty}I(M)=1$. A simple computation gives
\begin{align*}\|\varphi_N\|_{L^p(\mathbb{R})}&= 2^{1/p}(\log(N))^{1/p}, \\ \|\varphi_{\frac{N}{2}}\mathbf{1}_{(-\infty,-10{T}M]\cup[ 10{T}M,\infty)}\|_{L^p(\mathbb{R})}&= 2^{1/p}(\log(N)-\log(2)-\log(10{T}M))^{1/p}~.\end{align*}
Combining all of these together, we have justified that
\begin{align*}
   & \frac{\|\sup_{0<t<T}\mG_t^{d+1}(f \otimes \varphi_N)\|_{L^p(\mathbb{R}^{d+1})}}{\|f\otimes \varphi_N\|_{L^p(\mathbb{R}^{d+1})}}  = 
    \frac{\|\sup_{0<t<T}\mG_t^{d+1}(f\otimes \varphi_N)\|_{L^p(\mathbb{R}^{d+1})}}{\|f\|_{L^p(\mathbb{R}^d)} \| \varphi_N\|_{L^p(\mathbb{R})}}\\
    & \qquad \geq  \frac{\|\sup_{0<t<T} \mG_t^{d}f\|_{L^p(\mathbb{R}^{d})}}{\|f\|_{L^p(\mathbb{R}^d)} } \left( \frac{I(M)2^{1/p}(\log(N)-\log(2)-\log(10{T}M))^{1/p}}{2^{1/p}(\log(N))^{1/p}}\right)~.
\end{align*}
Taking the limit as $N\rightarrow\infty$ we obtain
$$\lim_{N\rightarrow\infty} \frac{\|\sup_{0<t<T}\mG_t^{d+1}(f \otimes \varphi_N)\|_{L^p(\mathbb{R}^{d+1})}}{\|f\otimes \varphi_N\|_{L^p(\mathbb{R}^{d+1})}} \geq \frac{\|\sup_{0<t<T}\mG_t^{d}f\|_{L^p(\mathbb{R}^{d})}}{\|f\|_{L^p(\mathbb{R}^d)} } I(M)~. $$
Finally, we let $M\rightarrow \infty$ and then we conclude by taking the limit for $T\rightarrow\infty$.
\qed 

Next, we compute lower bounds on $\|\mG_\ast\|_{L^p(\mathbb{R})}$ by testing $\mG_\ast$ on two natural examples, a homogeneous radial function and a Gaussian function.

\begin{lemma}
 For all $p\in (1,\infty)$ we have $\| \mathcal{G}_\ast \|_{L^p(\mathbb{R})} \geq \mG_\ast(|x|^{-1/p})(1) $. In particular
    $$\| \mathcal{G}_\ast \|_{L^p(\mathbb{R})} \geq \frac{2^{(p-1)/p}}{\sqrt{2e\pi}} \frac{p}{p-1}~.$$   
\end{lemma}
\proof It is easy to see that for $\varphi$ given by \eqref{eq: phi,phi_N} we have
$\mG_\ast(\varphi)(x)=\mG_\ast(\varphi)(1)|x|^{-1/p}$, $x\in \R.$ By standard approximation arguments (see e.g. \cite[Proof~of~Thm.~3.2]{DSS}), one also gets that $$\lim_{N\rightarrow\infty}\|\mG_\ast \varphi_N\|_{L^p(\mathbb{R})}/\|\varphi_N\|_{L^p(\mathbb{R})}\geq  \mG_\ast(|x|^{-1/p})(1)~.$$ 
A simple lower bound for $ \mG_\ast(|x|^{-1/p})(1)$ is 
\begin{align*}
 \mG_\ast(|x|^{-1/p})(1)&= \sup_{t>0}\frac{1}{\sqrt{2\pi t^2}}\int_{-\infty}^\infty |1-y|^{-1/p}e^{-y^2/(2t^2)}\dd y \\
    &\geq \sup_{t>0}\frac{1}{\sqrt{2\pi t^2}} e^{-1/(2t^2)}\int_{-1}^1 |1-y|^{-1/p} \dd y\\
    & \geq \frac{1}{\sqrt{2\pi e}} \frac{p}{p-1} 2^{(p-1)/p}~.
\end{align*}
\qed 

A better explicit lower bound is provided by the next example.

\begin{lemma}
     For all $p\in (1,\infty)$ we have
    $$\| \mathcal{G}_\ast \|_{L^p(\mathbb{R})} \geq  \left(\frac{2p}{\pi e^p}\right)^{1/(2p)} \left(\frac{p}{p-1}\right)^{1/p}~.$$   
\end{lemma}
\proof 
Set $\gamma(x)=\frac{1}{\sqrt{2\pi }}e^{-x^2/2}.$ Then
$$\sup_{t>0}\mG_t(\gamma)(x)=\sup_{t>0} \frac{1}{\sqrt{2\pi(1+t^2)}}e^{-x^2/(2(1+t^2))}~.$$
We set $u:=1+t^2$ and the condition $t>0$ translates into $u>1$. We define $h(u):=\frac{1}{\sqrt{2\pi u}}e^{-x^2/(2u)}$. Then we have $h'(u)=h(u)(-\frac{1}{2u}+\frac{x^2}{2u^2})$, $h'({x^2})=0,$ and $h''({x^2})<0$. Note that ${x^2}>1$ for $|x|>1$. Therefore, for $|x|>1$ it holds
$$\sup_{u>1} h(u)= \frac{1}{\sqrt{2\pi |x|^2}}e^{-1/2}~,$$
while for $|x|\leq 1$ we have
$$\sup_{u>1} h(u)= \frac{1}{\sqrt{2\pi }}e^{-x^2/2}~.$$
Using these facts, we obtain that
\begin{align*}
\|\mG_*(\gamma)\|_{L^p(\mathbb{R})}^p &= \frac{1}{(2\pi )^{p/2}}\int_{-1}^{1} e^{-x^2p/2}\dd x + \frac{2}{(2\pi e)^{p/2}}\int_{1}^\infty x^{-p}\dd x\\
& \geq \frac{2}{(2\pi  e)^{p/2}} \left( 1 + \frac{1}{p-1}\right)~.
\end{align*}
Combining with $\|\gamma\|_{L^p(\mathbb{R})}^p = \frac{1}{(2\pi )^{p/2}}\sqrt{\frac{2\pi }{p}}$, leads to
$$\|\mG_\ast\|_{L^p(\mathbb{R})}^p \geq \sqrt{\frac{2p}{\pi e^p}}\left( \frac{p}{p-1} \right)~.$$
\qed

 For $p\in(1,\infty)$ we set \[c_\sharp(p):=\left(\frac{2p}{\pi e^p}\right)^{1/(2p)} \left(\frac{p}{p-1}\right)^{1/p} \frac{p-1}{p}\]
and let
\[
c=\inf_{p\in (1,\infty)}c_\sharp(p).
\]
It follows from the previous lemma that 
$$\|\mG_\ast\|_{L^p(\mathbb{R})}\geq \max \left\lbrace \ c \frac{p}{p-1} , 1 \right\rbrace~.$$
This, combined with Proposition \ref{propo2:GausMonotonicity}, concludes the proof of Theorem \ref{propo:GausMonotonicity} apart from the lower bound $c>0.4$ which is the content of the lemma below. 
\begin{lemma}
    \label{lem: c>0.4}
The constant $c$ is larger than $0.4.$
\end{lemma}
\proof[Proof (sketch)]
The proof is standard, however, we provide details for the convenience of the reader. Denoting $x=(p-1)/p,$ $x\in (0,1)$ our goal is to show that
\[\inf_{x\in (0,1)} \bigg(\frac{2}{\pi(1-x)}\bigg)^{(1-x)/2}x^x>0.4\sqrt{e}.\]
This will follow if we show that the infimum $h_{\rm inf}$ of the function 
\[h(x)=\frac{(1-x)}{2}\big(\log (2/\pi)-\log(1-x))+x\log x,\qquad x\in (0,1)\]
is larger than $\log(0.4)+1/2.$ Since
\[h'(x)=\frac{1}{2}\log(1-x)+\log x+\frac{3}{2}-\frac{1}{2}\log(2/\pi)\]
we see that $h'(x)=0$ if and only if  $x\sqrt{1-x}=e^{-3/2}(2/\pi)^{1/2}$ which has two solutions $0<x_1<x_2<1.$ Furthermore $h''(x)=(1-3x/2)/(x(1-x))$ is positive for $x\in(0,2/3)$ and negative for $x\in (2/3,1)$ and thus $h$ achieves its infimum in $x_1,$ i.e.\ $h_{\rm inf}=h(x_1).$ One may also verify that $x_1\in (0.198,0.2)$ and thus using the fact that $x \log x$ decreases on $(0,0.3)$ we have 
\begin{align*}
h_{\rm inf}&= h(x_1)>\frac{(1-0.2)}{2}\big(\log (2/\pi)-\log(1-0.198))+0.2\log 0.2\\
& > -0.415>\log(0.4)+1/2
\end{align*}
as desired.
\qed

\section{ Concluding remarks - testing on radial functions}\label{sec:ConcludingRemarks}

We conclude by discussing three natural examples of testing functions previously considered in the literature. Namely, we consider as input a homogeneous radial function, the characteristic function of the unit ball, and a Gaussian function, and we study the behavior, as $d\rightarrow\infty$, of the norm on $L^p(\mathbb{R}^d)$ of the operators $\mathcal{S}_\ast(f),\,\mathcal{G}_\ast(f),\,\mathcal{B}_\ast(f)$ applied to these functions. Note that by Theorems \ref{thm_rad_logconc} and \ref{thm:gaussian_sphere} this behavior must be the same for any maximal function associated with a radial log-concave density. These three examples seem to suggest that the limit $\lm(p)$ may be equal to one when restricting to radial inputs only which we leave as Question \eqref{eq: limrad 1}.

Our first example is the function \[\varphi_d(x)=|x|^{-d/p},\qquad x\in \R^d\setminus\{0\}.\] This function is of interest because it is an eigenfunction of $\mathcal{B}_\ast$ (as well as of $\mS_\ast$, $\mG_\ast$, and, more in general, of any maximal function $\mK_\ast$ associated with a radial kernel). In fact, $\mathcal{B}_\ast \varphi_d (x)= b_{p,d} \varphi_d(x)$ with $b_{p,d}:=\mathcal{B}_\ast(|x|^{-d/p})(e_1)$ and $e_1:=(1,0,...,0)$. It has been shown in \cite[Eq.~3.1]{CG1} and  \cite[Theorem~3.2]{DSS} that  $\| \mathcal{B}_\ast \|_{L^p(\mathbb{R}^d)} \geq b_{p,d}$. Moreover, it has been shown in \cite[Lemma~3.4]{DSS} that $b_{p,d}>1$ whenever $p<\tfrac{d}{d-2}$ while it has been suggested in \cite[p.~478]{DSS} that $b_{p,d}$ may be identically one whenever $p\geq \tfrac{d}{d-2}$. In our analysis we will focus on $s_{p,d}:=\mS_*(|x|^{-d/p})(e_1),$ which is clearly larger than or equal to $b_{p,d}.$ Integration in polar coordinates shows that
\[
s_{p,d}=\sup_{r>0}
\frac{\omega_{d-2}}{\omega_{d-1}} \int_{-1}^{1} (1 - 2r x + r^2)^{-\frac{d}{2p}} (1 - x^2)^{\frac{d - 3}{2}} \, dx ~,
\]
where the symbol $\omega_{d-1}$ above denotes the surface area of the unit sphere in $\R^d.$
Let $p\in (1,\infty)$ be fixed and denote by $d_\sharp(p)$ the smallest $d\in\mathbb{N}$ for which $p\geq \frac{d}{d-2}$ holds.
\begin{lemma}
    Let $p\in (1,\infty)$ be fixed and take $d\geq \max \lbrace 3, d_\sharp(p) \rbrace.$ Then $s_{p,d}=1$.
\end{lemma}
\proof
In the considered range for $(p,d)$ the function $|x|^{-d/p}$ is superharmonic on $\mathbb{R}^d\setminus \lbrace 0 \rbrace$ beacuse \[\Delta(|x|^{-d/p})=\tfrac{d}{p}(2-d(p-1)/p)|x|^{-d/p-2}\leq 0.\] It follows from the properties of superharmonic functions that $\mathcal{S}_r(\varphi_d)(e_1)\leq \varphi_d(e_1)=1$ for all $r<1$. Define
$$g(r):=\int_{-1}^1(1-2rx+r^2)^{-\tfrac{d}{2p}}(1-x^2)^{\tfrac{d-3}{2}}\dd x$$
so that $\mathcal{S}_r(|x|^{-d/p})(e_1)=\tfrac{\omega_{d-2}}{\omega_{d-1}}\, g(r)$. We are interested in studying $\sup_{r>0}g(r)$. Observe that, since $\lim_{r\rightarrow\infty}g(r)=0$, the maximum is attained. We compute
$$g'(r)=-\frac{d}{2p}\int_{-1}^1(1-2rx+r^2)^{-\tfrac{d}{2p}-1}(2r-2x)(1-x^2)^{\tfrac{d-3}{2}}\dd x~.$$
We observe that $g'(1)<0$, so $r=1$  is not a maximum. Similarly, $g'(r)<0$ for all $r>1$ and therefore the maximum is not attained in this region. 
Then the maximum must be attained at some $r_p$ with $r_p<1$. It follows from the aforementioned properties of superharmonic functions that $g(r_p)\leq \frac{\omega_{d-1}}{\omega_{d-2}}$~. 
Hence, $\mathcal{S}_\ast(|x|^{-d/p})(e_1)\leq 1$ and, in view of the trivial bound $\varphi_d(x)\leq \mathcal{S}_\ast \varphi_d(x)$, the result in the statement follows.
\qed

As a consequence we have that, for $p\geq \frac{d}{d-2}$, $d\geq 3$, $b_{p,d}=1$ as suggested in \cite{DSS}.  This appears to be in line with \cite[Theorem~1]{K01} which asserts that $\mB_\ast$ has a non-constant fixed point if and only if $p\geq \frac{d}{d-2}$, $d\geq 3$. Moreover, \cite[Remark~1]{K01} observes that $f\in L^1_{loc}(\mathbb{R}^d)$ is a fixed point for $\mB_\ast$ if and only if $f$ is superharmonic. 

 Our second example concerns the characteristic function of the unit ball $B$ in $\R^d:$ 
 \[\chi_d(x)=\mathbf{1}_{B}(x),\qquad x\in \R^d.\] It has been observed in \cite[Remark 2.9]{APL} that for fixed $(p,d)$
\begin{equation*}
\frac{\|\mB_{*}\chi_d\|_{L^p(\R^d)}}{\|\chi_d \|_{L^p(\R^d)}}\ge \left(1+\frac{1}{2^{dp}(p-1)}\right)^{1/p} >1~,
\end{equation*}
hence providing the lower bound for $\|\mathcal{B}_\ast\|$ in equation \eqref{eq: mB bel APL}. The next lemma implies that
\[
\lim_{d\to \infty} \frac{\|\mB_{*}\chi_d\|_{L^p(\R^d)}}{\|\chi_d \|_{L^p(\R^d)}}=1.
\]

\begin{lemma}
\label{lem: spher ball}
 For each $p\in(1,\infty)$ we have 
  \begin{equation}
\label{eq: spher ball}
\lim_{d\rightarrow\infty}\frac{\|\mathcal{S}_\ast \chi_d\|_{L^p(\R^d)}}{\|\chi_d\|_{L^p(\R^d)}} = 1~.
\end{equation}
\end{lemma}
In the proof, we will use a formula for the surface area of a spherical cap on a sphere of radius $r.$ Namely, denoting by $\theta\in[0,\pi/2]$  the polar angle between the pole and the edge of the cap one has
\begin{equation}
\label{eq: Adr form}
A_{d,r}^{\text{cap}}(\sin(\theta))=\frac{1}{2}r^{d-1}\frac{2\pi^{d/2}}{\Gamma(\tfrac{d}{2})}I_{\sin^2(\theta)}\left(\frac{d-1}{2},\frac{1}{2}\right)~,\end{equation}
where $I_x(a,b)$ is the regularized incomplete beta function,
$$I_x(a,b):=\frac{\Gamma(a+b)}{\Gamma(a)\Gamma(b)}\int_0^x t^{a-1}(1-t)^{b-1}\dd t~,$$
see for example \cite{L11}.
\proof
We start with proving that 
\begin{equation}
\label{eq: S on chid}
\mathcal{S}_\ast \chi_d(x)= \chi_d(x)+\frac{1}{2} I_{(|x|^2+1)^{-1}}\left(\frac{d-1}{2},\frac{1}{2}\right)
\mathbf{1}_{\mathbb{R}^d\setminus B}(x).
\end{equation}
The formula is clear for $x\in B$ thus we may take $x\not \in B.$ Let $|(x+r{\mathbb{S}}^{d-1})\cap B|$ be the area of the spherical cap on the sphere of radius $r$ centered at $x$ which is created by intersecting this sphere with the unit ball. Note that this area is non-zero only for $|x|-1<r<|x|+1$. For such $r$ using \eqref{eq: Adr form} we obtain
\begin{equation}
\label{eq: mSr I}
\mS_r (\chi_d)(x)=\frac{|(x+r{\mathbb{S}}^{d-1})\cap B|}{r^{d-1}\sigma(\mathbb{S}^{d-1})}=\frac{1}{2} I_{\sin^2 \theta}\left(\frac{d-1}{2},\frac{1}{2}\right).
\end{equation}
where $\theta=\theta(r,x)$ is   the polar angle between the pole and the edge of the cap $(x+r{\mathbb{S}}^{d-1})\cap B.$ For $x\not \in B$ we want to maximize $\mS_r (\chi_d)(x)$ under the constraints $|x|-1<r<|x|+1.$ In view of \eqref{eq: mSr I} this boils down to determining the maximal value of $\sin\theta.$  Let $y$ be any point on $(x+r{\mathbb{S}}^{d-1})\cap \mathbb{S}^{d-1}.$ Then, $y,x,0$ form a triangle with side lengths $1,|x|,r$ and the angle $\theta$ between the sides of length $|x|$ and $r.$ Under our constraints $\sin \theta$ is the largest when the sides of length $1$ and $|x|$ are perpendicular in which case $\sin \theta=(|x|^2+1)^{-1/2}.$ This justifies \eqref{eq: S on chid}.

Now, using \eqref{eq: S on chid} we have 
\begin{align}\label{eq:max_on_indicator_fun_bound}
    \frac{\|\mathcal{S}_\ast \chi_d\|_{L^p(\R^d)}^p}{\|\chi_d\|_{L^p(\R^d)}^p}= 1 +\frac{1}{|B|}\int_{|x|\geq 1}\left|\frac{1}{2} I_{(|x|^2+1)^{-1}}\left(\frac{d-1}{2},\frac{1}{2}\right)\right|^p \dd x~.
\end{align}
For $|x|\geq 1$ can estimate the regularized beta function in the last display as
\begin{align*}
    I_{(|x|^2+1)^{-1}}\left(\frac{d-1}{2},\frac{1}{2}\right) & =\frac{\Gamma(\tfrac{d}{2})}{\Gamma(\tfrac{d-1}{2})\sqrt{\pi}} \int_0^{(|x|^2+1)^{-1}}t^{\tfrac{d-3}{2}}(1-t)^{-1/2}\dd t\\
    & \leq \frac{\Gamma(\tfrac{d}{2})}{\Gamma(\tfrac{d-1}{2})\sqrt{\pi}} \frac{1}{\left( 1- \tfrac{1}{|x|^2+1}\right)^{1/2}} \int_0^{(|x|^2+1)^{-1}}t^{\tfrac{d-3}{2}}\dd t \\
    & \leq  \frac{\Gamma(\tfrac{d}{2})\sqrt{2}}{\Gamma(\tfrac{d-1}{2})\sqrt{\pi}} \frac{2}{d-1}(|x|^2+1)^{-\tfrac{d-1}{2}}~.
\end{align*}
Hence, integrating in polar coordinates and using \eqref{eq:max_on_indicator_fun_bound} we get
\begin{align*}
     \frac{\|\mathcal{S}_\ast \chi_d\|_{L^p(\R^d)}^p}{\|\chi_d\|_{L^p(\R^d)}^p} & \leq 1 +\frac{\Gamma(\tfrac{d}{2}+1)}{\pi^{d/2}} \frac{2\pi^{d/2}}{\Gamma(\tfrac{d}{2})} \left( \frac{\Gamma(\tfrac{d}{2})}{\Gamma(\tfrac{d-1}{2})\sqrt{\pi}} \frac{\sqrt{2}}{d-1}\right)^p \int_1^\infty (r^2+1)^{-\tfrac{p}{2}(d-1)} r^{d-1}\dd r\\
     & \leq 1 +\frac{2\Gamma(\tfrac{d}{2}+1) }{\Gamma(\tfrac{d}{2})} \left( \frac{\Gamma(\tfrac{d}{2})}{\Gamma(\tfrac{d-1}{2})\sqrt{\pi}} \frac{\sqrt{2}}{d-1}\right)^p\frac{1}{d(p-1)-p}~.
\end{align*}
Using Stirling's formula \eqref{eq: Stirling} we see that the limit as $d\rightarrow\infty$ of the quantity on the right-hand-side of the last display is equal to one. In view of the trivial lower bound $\|\mathcal{S}_\ast\chi_d\|_{L^p(\mathbb{R}^d)}/\|\chi_d\|_{L^p(\mathbb{R}^d)}\geq 1$ the proof of \eqref{eq: spher ball} is completed.
\qed

Our last example covers the case of the Gaussian input function
\[
\gamma_{d}(x)=\frac{1}{(2\pi )^{d/2}}e^{-|x|^2/2},\qquad x\in \R^d.
\]
\begin{lemma}\label{lem: gaus gaus}
    For each  $p\in(1,\infty)$ we have 
\[
\lim_{d\rightarrow\infty}\frac{\|\mG_\ast \gamma_{d}\|_{L^p(\R^d)}}{\|\gamma_{d}\|_{L^p(\R^d)}}=1.
\] 
\end{lemma}
\noindent Clearly, in view of Theorems \ref{thm_rad_logconc} and \ref{thm:gaussian_sphere}, an immediate consequence of the lemma is that
\[\lim_{d\rightarrow\infty}  \frac{\|\mS_\ast\gamma_d\|_{L^p(\mathbb{R}^d)}}{\|\gamma_d\|_{L^p(\mathbb{R}^d)}}= \lim_{d\rightarrow\infty} \sup_{K\in\LCrad} \frac{\|\mK_\ast\gamma_d\|_{L^p(\mathbb{R}^d)}}{\|\gamma_d\|_{L^p(\mathbb{R}^d)}}=1~.\]

\proof Since $\gamma_d$ is a Gaussian we have 
$$\sup_{t>0} \mG_t (\gamma_d)(x)=\sup_{t>0}\frac{1}{(2\pi(1+t^2))^{d/2}}e^{-|x|^2/(2(1+t^2))}~.$$
We set $u:=1+t^2$ which translates the condition $t>0$ into $u>1$. Note that the function $h(u):= \frac{1}{(2\pi u)^{d/2}}e^{-|x|^2/(2u)}$
satisfies $h'(u)=h(u)(-\frac{d}{2u}+\frac{|x|^2}{2u}),$ $h'(\frac{|x|^2}{d})=0$ and $h''(\frac{|x|^2}{d})\leq 0$. Since $\frac{|x|^2}{d}>1$ for $|x|>\sqrt{d}$ we see that
$$\sup_{u>1} h(u)=\left( \frac{d}{2\pi|x|^2}\right)^{d/2}e^{-d/2} \qquad \text{for}\; |x|>\sqrt{d}~,$$
while for $|x|\leq \sqrt{d}$ the supremum is attained at the boundary $u=1$, namely
$$\sup_{u>1} h(u)=\frac{1}{(2\pi )^{d/2}}e^{-|x|^2/2}=\gamma_d(x) \qquad \text{for}\; |x|\leq \sqrt{d}~.$$
Using these facts, we obtain that
\begin{align*}
    \|\sup_{t>0} \mG_t (\gamma_d)\|_{L^p(\mathbb{R}^d)}^p &= \int_{|x|\leq \sqrt{d}} |\gamma_d(x)|^p\dd x + \left(\frac{d}{2\pi e} \right)^{dp/2}\int_{|x|> \sqrt{d}} |x|^{-dp}\dd x \\
    & =\int_{|x|\leq \sqrt{d}} |\gamma_d(x)|^p\dd x  + \left(\frac{d}{2\pi e} \right)^{dp/2} \frac{2\pi^{d/2}}{\Gamma(\frac{d}{2})} \frac{1}{d^{\frac{d}{2}(p-1)}} \frac{1}{d(p-1)}~.
\end{align*}
We compute $\|\gamma_d\|_{L^p(\mathbb{R})^d}^p=\frac{1}{(2\pi )^{dp/2}}(\frac{2\pi }{p})^{d/2}$. Hence, to conclude, we are left to show that, for any fixed $p\in(1,\infty)$, the following limit is equal to zero,
\begin{align*}
    \lim_{d\rightarrow\infty} \;\; (2\pi )^{dp/2} \left(\frac{p}{2\pi }\right)^{d/2}&\left( \left(\frac{d}{2\pi e} \right)^{dp/2} \frac{2\pi^{d/2}}{\Gamma(\frac{d}{2})} \frac{1}{d^{\frac{d}{2}(p-1)}} \frac{1}{d(p-1)} \right)\\
    & = \lim_{d\rightarrow\infty} \left(\frac{1}{e}\right)^{dp/2}\left( \frac{pd}{2}\right)^{d/2}\frac{2}{\Gamma(\frac{d}{2})}\frac{1}{d(p-1)} =: \lim_{d\rightarrow\infty} v(p,d)~.
\end{align*}
We observe that for any fixed $d\geq 1$, $v(p,d)$ is a decreasing function of  $p\in (1,\infty)$. Therefore, for any fixed $p$ and sufficiently small, fixed, $\varepsilon>0$, using Stirling's approximation \eqref{eq: Stirling} we have 
\begin{align*}
   \lim_{d\rightarrow\infty} v(p,d)\leq  \lim_{d\rightarrow\infty}
   \left(\frac{1}{e}\right)^{(1+\varepsilon)d/2}\left( \frac{(1+\varepsilon)d}{2}\right)^{d/2}\frac{2}{\Gamma(\frac{d}{2})}\frac{1}{d\varepsilon} = 0~,
\end{align*}
as claimed.
\qed

\begin{remark} Let $\Prad$ be the family of radial probability measures on $\R^d$. For $M\in \Prad$ we write
\[
\mM_t f(x)=\int_{\R^d} f(x-ty)\,dM(y),\qquad  \mM_*f(x)=\sup_{t>0}|\mM_t f(x)|.
\]
Clearly, this is in accord with \eqref{eq:mKt} when $M$ has a density, i.e. $dM(y)=K(y)dy$.
In view of the pointwise inequality
\begin{equation*}
\mathcal{M}_\ast f(x)\le \mathcal{S}_\ast f(x) ~,
\end{equation*}
and of the trivial lower bound ${\|\mathcal{M}_\ast f\|_{L^p(\R^d)}}/{\|f\|_{L^p(\R^d)}}\geq 1$, Lemmas \ref{lem: spher ball} and \ref{lem: gaus gaus} imply that
\[ \lim_{d\rightarrow\infty}\sup_{M\in \Prad} \frac{\|\mathcal{M}_\ast \chi_d\|_{L^p(\R^d)}}{\|\chi_d\|_{L^p(\R^d)}} = \lim_{d\rightarrow\infty}\sup_{M\in \Prad} \frac{\|\mathcal{M}_\ast \gamma_d\|_{L^p(\R^d)}}{\|\gamma_d\|_{L^p(\R^d)}}=1~,
   \]
   where now the supremum is over the family of \textit{all} radial probability measures on $\R^d$.
   \end{remark}

\subsection*{Acknowledgments}
The authors are grateful to Tony Carbery for literature references.

\appendix

\section{Auxiliary lemmas for Bessel functions}
\label{sec: auxlem}
The purpose of this appendix is to collect some useful identities and estimates involving Bessel functions that are needed in the proof of Theorem \ref{thm:gaussian_sphere}. First, we recall the integral representation 
\begin{align}\label{integral_repr_bessel}
    J_\nu(r)= \frac{\big(\tfrac{r}{2}\big)^{\nu}}{\Gamma(\nu+\tfrac{1}{2})\sqrt{\pi}} \int_{-1}^1 e^{irs}(1-s^2)^{(2\nu-1)/2}
ds~, \quad \nu>-\frac{1}{2},\, r\geq 0~, \end{align}
see e.g. \cite[Appendix~B]{Gr14}.
Moreover, as a consequence of the recursion formulas 
\begin{align*}
   \frac{2\nu}{r} J_\nu(r) & = J_{\nu-1}(r)+J_{\nu+1}(r) \\\
   2 J'_\nu(r) & = J_{\nu-1}(r) - J_{\nu+1}(r)
\end{align*}
we have the identity
\begin{align}\label{first_prop_bessel}
J_\nu'(r)=J_{\nu-1}(r)-\frac{\nu}{r}J_\nu(r).
\end{align}

\begin{lemma}
Let $r>0$. The following identities hold,
\begin{align}\label{first_identity_rdmudr}
    r\frac{\dd}{\dd r} \mu(r)= 2 \frac{\Gamma(\tfrac{d}{2})}{\pi^{d/2-1}}\bigg(-\frac{d}{2}+1 \bigg) r^{-\tfrac{d}{2}+1}J_{\tfrac{d}{2}-1}(2\pi r) + 2\pi \frac{\Gamma(\tfrac{d}{2})}{\pi^{d/2-1}}r^{-\tfrac{d}{2}+2}J_{\tfrac{d}{2}-2}(2\pi r)~,
\end{align}
\begin{align}\label{second_identity_rdmudr}
     r\frac{\dd}{\dd r} \mu(r)= r (i2\sqrt{\pi}) \frac{\Gamma(\tfrac{d}{2})}{\Gamma(\tfrac{d}{2}-\tfrac{1}{2})} \int_{-1}^1 e^{i2\pi rs}s(1-s^2)^{(d-3)/2}\dd s~.
\end{align}
\end{lemma}
\proof Identity \eqref{first_identity_rdmudr} can be obtained by direct computation using the definition of $\mu$  together with identity \eqref{first_prop_bessel}. Identity \eqref{second_identity_rdmudr} can be obtained by direct computation using the definition of $\mu$ and the integral representation formula for Bessel functions \eqref{integral_repr_bessel}.
\qed

The following lemma provides a bound for the integral in the right-hand-side of \eqref{second_identity_rdmudr}.
\begin{lemma}
    There exists a constant $c>0$ independent of the dimension $d$ such that for all $d\geq 2$, $r\ge 0$, the following pointwise estimate holds
    \begin{align}\label{estimate_as_in_MSW}
       \bigg| \int_{-1}^1 e^{i2\pi rs}s(1-s^2)^{(d-3)/2}\dd s \bigg| \lesssim \frac{e^{-2\pi r/\sqrt{d}}}{d}+\frac{e^{-d/10}}{\sqrt{d}}~.
    \end{align}
\end{lemma}
\proof The statement is easily seen to be true for small values of $d$ hence we consider only sufficiently large values of $d\in\mathbb{N}$. We follow with very minor modifications the arguments in the proof of \cite[Lemma~4.1]{MSW24}, see also \cite[Lemma~3.6]{KW23}. We start by estimating
\begin{align*}
     &\bigg| \int_{-1}^1  e^{i2\pi rs}s(1-s^2)^{(d-3)/2}\dd s \bigg| = \frac{1}{\sqrt{d}}\bigg| \int_{-\sqrt{d}}^{\sqrt{d}} e^{i2\pi rs/ \sqrt{d}}\frac{s}{\sqrt{d}}(1-\frac{s^2}{d})^{(d-3)/2}\dd s \bigg|\\
          &\leq \frac{1}{\sqrt{d}}\bigg| \int_{\tfrac{\sqrt{d}}{2}\leq |s|\leq \sqrt{d}} e^{i2\pi rs/\sqrt{d}}\frac{s}{\sqrt{d}}\bigg(1-\frac{s^2}{d}\bigg)^{(d-3)/2}\dd s \bigg| + \frac{1}{\sqrt{d}}\bigg| \int_{|s|\leq \tfrac{\sqrt{d}}{2}} e^{i2\pi rs/ \sqrt{d}}\frac{s}{\sqrt{d}}\bigg(1-\frac{s^2}{d}\bigg)^{(d-3)/2}\dd s \bigg|~.
\end{align*}
Using the fact that $1-\frac{s^2}{d}\leq \frac{3}{4}$ for $|s|\geq \frac{\sqrt{d}}{2}$ we bound the first integral in the last display by
\begin{align*}
   \bigg| \int_{\tfrac{\sqrt{d}}{2}\leq |s|\leq \sqrt{d}} e^{i2\pi rs/\sqrt{d}}\frac{s}{\sqrt{d}}\bigg(1-\frac{s^2}{d}\bigg)^{(d-3)/2}\dd s \bigg|\leq \sqrt{d}\bigg( \frac{3}{4}\bigg)^{\tfrac{d-3}{2}}\lesssim e^{-d/10}~.
\end{align*}
To estimate the second integral
     \begin{align*}
      \bigg| \int_{|s|\leq \tfrac{\sqrt{d}}{2}} e^{i2\pi rs/ \sqrt{d}}\frac{s}{\sqrt{d}}\bigg(1-\frac{s^2}{d}\bigg)^{(d-3)/2}\dd s \bigg|~
      \end{align*}
we change the contour of integration. Let $\gamma:=\gamma_0\cup\gamma_1\cup\gamma_2\cup\gamma_3$ where
\begin{align*}
    &\gamma_0(s):=s \qquad &\text{for} \; s\in[-\tfrac{\sqrt{d}}{2},\tfrac{\sqrt{d}}{2}]~,\\
    &\gamma_1(s):=is+\frac{\sqrt{d}}{2} \qquad &\text{for} \; s\in[0,1]~,\\
    &\gamma_2(s):=-s+i \qquad &\text{for} \; s\in[-\tfrac{\sqrt{d}}{2},\tfrac{\sqrt{d}}{2}]~,\\
     &\gamma_3(s):=i(1-s)-\frac{\sqrt{d}}{2} \qquad &\text{for} \; s\in[0,1]~.
\end{align*}
As $z\mapsto e^{i2\pi rz/ \sqrt{d}}\tfrac{z}{\sqrt{d}}(1-\tfrac{z^2}{d})^{(d-3)/2}$ is holomorphic in $\lbrace z\in\mathbb{C}:\, |z|<\sqrt{d}\rbrace$, by Cauchy integral theorem we have
\begin{align*}
     \bigg| \int_{|s|\leq \tfrac{\sqrt{d}}{2}} e^{i2\pi rs/ \sqrt{d}}\frac{s}{\sqrt{d}}\bigg(1-\frac{s^2}{d}\bigg)^{(d-3)/2}\dd s \bigg| \leq & \sum_{k\in\lbrace 1,3 \rbrace} \bigg| \int_0^1 e^{i2\pi r\gamma_k(s)/ \sqrt{d}}\frac{\gamma_k(s)}{\sqrt{d}}\bigg(1-\frac{\gamma_k(s)^2}{d}\bigg)^{(d-3)/2}\dd s \bigg|\\
    & + \bigg| \int_{|s|\leq \tfrac{\sqrt{d}}{2}} e^{i2\pi r(-s+i)/ \sqrt{d}}\frac{(i-s)}{\sqrt{d}}\bigg(1-\frac{(i-s)^2}{d}\bigg)^{(d-3)/2}\dd s \bigg|~.  
\end{align*}
We bound the first term in the right-hand-side of the last display as
\begin{align*}
    \sum_{k\in\lbrace 1,3 \rbrace} \bigg| \int_0^1 e^{i2\pi r\gamma_k(s)/ \sqrt{d}}\frac{\gamma_k(s)}{\sqrt{d}}\bigg(1-\frac{\gamma_k(s)^2}{d}\bigg)^{(d-3)/2}\dd s \bigg|\leq 2 \bigg(\frac{1}{\sqrt{d}}+\frac{1}{2}\bigg)\bigg(\frac{3}{4}+\frac{1}{d}+\frac{1}{\sqrt{d}}\bigg)^{(d-3)/2}\lesssim e^{-d/10}~.
\end{align*}
As $e^{i2\pi r(-s+i)/ \sqrt{d}}=e^{-2\pi r/\sqrt{d}}e^{-i2\pi rs/\sqrt{d}}$, to conclude it is enough to show that
\begin{align}\label{estimate_last_term_after_contour}
     \int_{|s|\leq \tfrac{\sqrt{d}}{2}} |i-s|\bigg|1-\frac{(i-s)^2}{d}\bigg|^{(d-3)/2}\dd s \lesssim 1~.
\end{align}
Note that
\begin{align*}
   \bigg| 1- \frac{(s-i)^2}{d} \bigg|\leq 1+\frac{1-s^2}{d}+\frac{2|s|}{d} \leq \begin{cases}
 1+\frac{6}{d} & \text{for} \; |s|\leq \frac{5}{2}\\
 1-\frac{s^2}{36d} & \text{for} \; \frac{5}{2}\leq |s| \leq \frac{\sqrt{d}}{2}\\
    \end{cases} 
\end{align*}
and
\begin{align*}
 \int_{|s|\leq 1} |s|(1-s^2)^{(d-3)/2}\dd s= \frac{2}{d-1}~.
 \end{align*}
Using these facts, we see that the left-hand-side of \eqref{estimate_last_term_after_contour} is bounded by a universal constant times
 \begin{align*}
      & \int_{|s|\leq \tfrac{5}{2}} \bigg( 1+\frac{6}{d} \bigg)^{(d-3)/2}\dd s +  \int_{\tfrac{5}{2}<|s|\leq \tfrac{\sqrt{d}}{2}} |s|\bigg( 1-\frac{s^2}{36d} \bigg)^{(d-3)/2}\dd s  
      \lesssim  1+  \int_{|s|\leq \tfrac{\sqrt{d}}{2}} |s| \bigg(1-\frac{s^2}{36d}\bigg)^{(d-3)/2}\dd s \\
        \lesssim & 1+ d\int_{|s|\leq 1} |s|(1-s^2)^{(d-3)/2}\dd s \lesssim 1+   \frac{d}{d-1} 
        \lesssim  1~,
 \end{align*}
 hence concluding the proof.
 \qed

\end{document}